\documentclass{amsart}

\usepackage{amsmath}
\usepackage{amsthm} 
\usepackage{graphicx}
\usepackage{color}
\usepackage{scalerel}

\usepackage{enumerate}
\usepackage{amsfonts}
\usepackage{amssymb}
\usepackage[hidelinks]{hyperref}
\usepackage[capitalize]{cleveref}

\usepackage{nicefrac}
\usepackage{xfrac}

\usepackage[normalem]{ulem}

\usepackage{mathtools}

\newcommand{\e}{\varepsilon}
\newcommand{\tp}{\mathrm{tp}}

\newcommand{\Lc}{\mathcal{L}}
\newcommand{\Uc}{\mathcal{U}}

\newcommand{\Ac}{\mathcal{A}}
\newcommand{\Dc}{\mathcal{D}}
\newcommand{\Cc}{\mathcal{C}}

\newcommand{\Rb}{\mathbb{R}}

\newcommand{\Eb}{\mathbb{E}}
\newcommand{\frk}{\mathfrak}

\newcommand{\bbar}{\bar{b}}
\newcommand{\ybar}{\bar{y}}
\newcommand{\xbar}{\bar{x}}
\newcommand{\zbar}{\bar{z}}

\newcommand{\hbarr}{\bar{h}}

\def\Av{\operatorname{Av}}

\def\tp{\operatorname{tp}}
\def\KM{\operatorname{KM}}

\def\Aut{\operatorname{Aut}}
\def\Sym{\operatorname{Sym}}
\def\sym{\operatorname{sym}}
\def\fSym{\operatorname{Sym}_{\operatorname{fin}}}

\providecommand{\dotdiv}{
  \mathbin{
    \vphantom{+}
    \text{
      \mathsurround=0pt 
      \ooalign{
        \noalign{\kern-.35ex}
        \hidewidth$\smash{\cdot}$\hidewidth\cr 
        \noalign{\kern.35ex}
        $-$\cr 
      }%
    }%
  }%
}

\newcommand{\cB}{\mathcal{B}}

\newcommand{\cU}{\mathcal{U}}
\newcommand{\cL}{\mathcal{L}}

\newcommand{\R}{\mathbb{R}}
\newcommand{\seq}{\subseteq}
\newcommand{\kM}{\mathfrak{M}}
\newcommand{\abar}{\bar{a}}
\newcommand{\cbar}{\bar{c}}

\newcommand{\bfx}{\boldsymbol{x}}
\newcommand{\bfy}{\boldsymbol{y}}

\newcommand{\tsup}[1]{{\textstyle\sup_{#1}}}
\newcommand{\tinf}[1]{{\textstyle\inf_{#1}}}
\newcommand{\tmax}[1]{{\textstyle\max_{#1}}}
\newcommand{\tmin}[1]{{\textstyle\min_{#1}}}

\newtheorem{thm}{Theorem}[section]
\newtheorem{theorem}[thm]{Theorem}
\newtheorem{prop}[thm]{Proposition}
\newtheorem{proposition}[thm]{Proposition}
\newtheorem{lem}[thm]{Lemma}
\newtheorem{lemma}[thm]{Lemma}

\newtheorem{corollary}[thm]{Corollary}
\newtheorem{fact}[thm]{Fact}

\theoremstyle{definition}
\newtheorem{question}[thm]{Question}

\newtheorem{definition}[thm]{Definition}
 
\newtheorem{remark}[thm]{Remark}
\newtheorem*{claim-star}{Claim}

\newcommand{\clqedsym}{\dashv_{\text{\scriptsize claim}}}

\newcommand{\clqed}{%
  \ifmmode
    \eqno{\clqedsym}%
  \else
    \hfill$\clqedsym$%
  \fi
}

\def\Ind{\setbox0=\hbox{$x$}\kern\wd0\hbox to 0pt{\hss$\mid$\hss}
\lower.9\ht0\hbox to 0pt{\hss$\smile$\hss}\kern\wd0}

\def\Notind{\setbox0=\hbox{$x$}\kern\wd0\hbox to 0pt{\mathchardef
\nn=12854\hss$\nn$\kern1.4\wd0\hss}\hbox to
0pt{\hss$\mid$\hss}\lower.9\ht0 \hbox to 0pt{\hss$\smile$\hss}\kern\wd0}

\newcommand{\dotminus}{%
  \!\buildrel\textstyle .\over{%
    \hbox{\vrule height3pt depth0pt width0pt}{\smash{-}}%
  }%
}

\newcommand{\inv}{^{\text{-}1}}

\newcommand{\fim}{\textit{fim}}

\renewcommand{\epsilon}{\varepsilon}

\allowdisplaybreaks 

\begin{document}

\title{Generically stable Keisler measures}

\date{September 16, 2026}
\author[G. Conant]{Gabriel Conant}

\address{Department of Mathematics\\
University of Illinois Chicago\\
Chicago, IL 60607, USA}
\email{gconant@uic.edu}%

\author[K. Gannon]{Kyle Gannon}

\address{Beijing International Center for Mathematical Research (BICMR)\\
Peking University\\
No. 5 Yiheyuan Road, Haidian District, Beijing, China}
\email{kgannon@bicmr.pku.edu.cn}

\author[J. Hanson]{James E. Hanson}

\address{Department of Mathematics \\
  Iowa State University \\
  396 Carver Hall \\
  411 Morrill Road \\
  Ames, IA 50011, USA}
\email{jameseh@iastate.edu}

\thanks{GC was partially supported by NSF grant DMS-2452816. KG was partially supported by the Fundamental Research Funds for the Central Universities, Peking University, grant no.\ 7100604835 and by the National Natural Science Foundation of China, grant no.\ 12501001. JH was partially supported by NSF grant 2554117.}

\maketitle

\vspace{-10pt}

\begin{abstract}
Given a first-order theory $T$ (in discrete or continuous logic) and a Borel-definable global Keisler measure $\mu$ in $T$, we show that the following conditions are equivalent: $(i)$ $\mu$ is a frequency interpretation measure; $(ii)$ $\mu$ is definable and its canonical ``random extension" $r_\mu$   is generically stable in the randomization theory $T^R$; $(iii)$ $\mu$ is ``self-averaging". This result establishes a robust notion of generic stability for Keisler measures, which resolves a long-term research objective from  previous work. The implications $(i)\Rightarrow(ii)\Rightarrow (iii)$ were previously established by the authors (for $T$ discrete). The primary focus of this paper is  the reverse implications $(iii)\Rightarrow (ii)\Rightarrow(i)$. We also prove that generically stable measures are closed under Morley products, answering another well-known question that  was open  even in the case of types. These results are obtained through the use of AI models.
\end{abstract}

 \section{Introduction}

  In \cite{CoGa}, the first two authors observed that a global type in a first-order theory is generically stable (as defined by Pillay and Tanovi\'{c} \cite{PiTa}) if and only if the corresponding Dirac measure  is a \emph{frequency interpretation measure} (\emph{fim}) (as defined by Hrushovski, Pillay, and Simon \cite{HPS} in their work on NIP theories). This motivated an extended investigation into the question of whether \emph{fim} is the ``right" analogue of generic stability for Keisler measures. Recently in \cite{CGH2}, we proposed two other such analogues and we proved a chain of implications between these notions and \emph{fim}. In this paper, we will  finish the task of showing that all three notions are  equivalent, which we view as a definitive resolution of this question.

We work with a  complete  theory $T$ in discrete or continuous logic with monster model $\cU$. 
 The randomization $T^R$ of $T$ will play a central role in our results. Given an atomless probability algebra $\Omega$ and a model $M$ of $T$, we let $M^{\Omega}\models T^R$ denote the randomization generated by measurable maps from $\Omega$ into $M$ with finite image. As in \cite{CGH2}, we let $\Cc$ be a monster model of $T^R$, with $\cU^\Omega\prec\Cc$. Given a   Keisler measure $\mu\in\kM_x(\cU)$, if $\mu$ is definable over $M\prec\cU$ then,  by a result of Ben Yaacov \cite{BYT},  there is a unique  $M^\Omega$-definable type $r_\mu\in S_x(\Cc)$, which is a ``canonical extension" of  $\mu$ in a  sense made precise below. We now state the first main result of this paper.
 
 \begin{theorem}\label{thm:main}
 Suppose $\mu\in\kM_x(\cU)$ is Borel-definable over $M\prec\cU$. Then the following are equivalent.
 \begin{enumerate}[$(i)$]
 \item $\mu$ is a frequency interpretation measure (fim) over $M$.
 \item $\mu$ is definable over $M$ and $r_\mu$ is generically stable over $M^\Omega$. 
 \item $\mu$ is self-averaging over $M$.
 \end{enumerate}
 \end{theorem}

 The definitions of generic stability for types and self-averaging for measures will be recalled below (see Definitions \ref{def:gs} and \ref{def:SA}, respectively). On the other hand, the precise definition of frequency interpretation measures will not be  needed for any results here, and thus we omit it and  refer the reader to \cite{HPS,CGH2,AndGS}. We also note that in \cite{Khanaki}, $\mu$ is called ``independently randomly generically stable" (\emph{irgs}) if it satisfies condition $(ii)$.

The implications $(i)\Rightarrow(ii)\Rightarrow(iii)$ in Theorem \ref{thm:main} were proved by the authors in \cite{CGH2} for discrete $T$. We further proved that $(iii)$ implies $\mu$ is finitely approximable (\emph{fam}) over $M$ (so, in particular, definable), and we verified that for types, $(i)$, $(ii)$, and $(iii)$ are all equivalent to generic stability (as originally defined in \cite{PiTa}). But we were unable to reverse either $(i)\Rightarrow (ii)$ or $(ii)\Rightarrow (iii)$ for arbitrary measures.

 Here we will extend $(i)\Rightarrow(ii)\Rightarrow(iii)$ to continuous $T$, and also prove $(iii)\Rightarrow (ii)\Rightarrow (i)$ for $T$ either discrete or continuous. It turns out that $(ii)\Rightarrow (i)$ follows from the main results in \cite{CGH2} using a straightforward coding trick. On the other hand, the proof of $(iii)\Rightarrow (ii)$ requires  substantial new arguments, which draw from established results in probability theory and dynamics on convex sets \cite{HewSav, deLaRue, Ryzh}.

The proposal to use \emph{fim} as the definition of generic stability for Keisler measures was tentatively suggested in \cite[Section 3]{CoGa} and \cite[Section 6]{CoGaHa}, and implicitly adopted in \cite[Section 3]{CoGaHa}. 
In light of Theorem \ref{thm:main}, we now make this definition explicit. 

\begin{definition}\label{def:gsm}
Suppose $\mu\in\kM_x(\cU)$ is Borel-definable over $M\prec\cU$. Then $\mu$ is \textbf{generically stable over $M$} if it satisfies the equivalent properties of Theorem \ref{thm:main}.
\end{definition}

 We now state the second main result of this paper.

 \begin{theorem}\label{thm:MP}
 Suppose $\mu\in\kM_x(\cU)$ and $\nu\in\kM_y(\cU)$ are generically stable over $M\prec\cU$. Then $\mu\otimes\nu$ is generically stable over $M$.
 \end{theorem}

 Despite several previous attacks, this problem has remained open for arbitrary theories even in the case of types. We will describe the history and previous progress at the start of Section \ref{sec:MP}. To prove Theorem \ref{thm:MP}, we will first separately handle the case of types (see Theorem \ref{thm:MPtypes}). The general result for measures will then follow using  the characterization of generic stability provided by Theorem \ref{thm:main}$(ii)$.

\subsection*{Outline of the paper}
In Section \ref{sec:rando}, we recall the basic setup of randomizations of continuous theories and Ben Yaacov's transfer map. 
The proof of Theorem \ref{thm:main} is given in Sections \ref{sec:rgsfim} and \ref{sec:sargs}. In particular, $(i)\Rightarrow (ii)$ is Corollary \ref{cor:fimtorgs},  $(ii)\Rightarrow(i)$ is Theorem \ref{thm:rgstofim}, $(ii)\Rightarrow (iii)$ is Corollary \ref{cor:rgstoSA}, and $(iii)\Rightarrow (ii)$ is Theorem \ref{thm:SAtorgs}. Theorem \ref{thm:MP} is then proved in Section \ref{sec:MP} (although we note that the results of Section \ref{sec:sargs} are not needed for this proof). Section \ref{sec:more} contains some further results and remarks.  Theorem \ref{thm:COP} characterizes generic stability for measures in terms of an order property condition reminiscent of standard results for types. In Subsection \ref{sec:rgs} we discuss a recent related preprint of Khanaki \cite{Khanaki}. Finally, in Appendix \ref{sec:appendix}, we provide a direct translation of several results from \cite{CGH2} for discrete $T$ to the case where $T$ is continuous. 

\subsection*{AI Acknowledgment} 
A proof of Theorem \ref{thm:main}[$(iii)\Rightarrow(ii)\Rightarrow(i)$] in the discrete case was initially obtained from a ChatGPT 5.5 query. We were also able to independently find proofs  using Kimi K3 and Claude Fable 5. These arguments were heavily reorganized and rewritten by the authors with further assistance from ChatGPT 5.6 Sol. Theorem \ref{thm:COP} was obtained by the authors by modifying a different result found by ChatGPT 5.6 Sol while attempting to answer  Question \ref{ques:OP} (see further remarks in Subsection \ref{sec:probs}). In the first version of this paper, Theorem \ref{thm:MP} was stated as a question, which  eluded the authors and ChatGPT for some time. An initial query with ChatGPT 6 Astra remained unsuccessful. Motivated by discussion with the first author, Caroline Terry attempted several prompt approaches with Astra, one of which successfully resulted in a  proof for types in discrete $T$, which was based on our previous work on ``stable \emph{ict}-patterns" \cite{CGH2}. After this, the third author separately obtained a more direct proof from Astra with a goal prompt in the Codex CLI. This is the proof that we have simplified and rewritten here. Nevertheless, we thank Caroline Terry for her persistence.

 \section{Randomizations and the transfer map}\label{sec:rando}
 Let $\mathcal{L}$ be a language in continuous logic and let $T$ be a complete first-order $\mathcal{L}$-theory.
For the most part, we will omit standard definitions regarding Keisler measures and Morley products. The reader is referred to previous papers by the authors such as \cite{CoGa}, \cite{CoGaHa} and \cite{CGH2} (which are written for discrete theories), as well as work of Anderson \cite{AndGS} where  many of these basic tools are adapted to continuous logic. Specific references will be given below when necessary. Unless otherwise stated, we will follow \cite{CGH2} for all definitions and notation. In fact, the reader should treat this paper as a direct continuation of  \cite{CGH2}.

Given $M\models T$, we let $\kM_x(M)$ denote the space of Keisler measures over $M$ in variables $x$. 
We will tacitly use the fact that definability of Keisler measures is preserved by Morley products, and that associativity holds for Morley products of definable Keisler measures. This is proved for discrete $T$ in  \cite[Proposition 2.6]{CoGa} and adapted to continuous $T$ in \cite[Lemma 2.9]{AndGS}.

  We let $T^{R}$ denote the theory of the randomization of $T$. 
   We briefly describe the construction of a model $M^\Omega\models T^R$ from a model $M\models T$, where $\Omega=(\Omega,\mathcal{B},\mathbb{P})$ is an atomless probability space. First, $M_0^\Omega$ is defined to be the set of $\cB$-measurable functions from $\Omega$ to $M$ with finite image, equipped with the pseudometric $d^\Omega(f,g)=\int_{t \in \Omega} d(f(t),g(t))\,d\mathbb{P}$, where $d$ is the underlying metric on $M$. Then $M^\Omega$ is the metric completion of  $M_0^\Omega$. 

   It is worth noting that the theory $T^R$ is formally in a two-sorted language. The above description of models of $T^R$ only explicitly describes the primary ``function sort", while the second auxiliary sort is occupied by the space of $\mathcal{B}$-measurable functions from $\Omega$ to $[0,1]$. See \cite{BYRV} for details.\footnote{When $T$ is discrete, the second sort is usually presented as just the algebra $\cB$; see \cite{BYKR}.}
     As in \cite{CGH2}, we will for the most part suppress this second  sort.  
   
   Next we recall the  interpretation of basic formulas in models of $T^R$. In particular, for any $M\models T$ and  $\mathcal{L}$-formula $\varphi(x_1,\ldots,x_{n})$, the $\cL^R$-formula  $\Eb[\varphi(x_1,\ldots,x_{n})]$ is evaluated on tuples $\bar{h} = (h_1,\ldots,h_n)$ from $M^{\Omega}_0$ via 
\[
    \Eb[\varphi(\bar{h})] = \int_\Omega \varphi(h_1(t),\ldots,h_n(t))\,d\mathbb{P}(t), 
\]
and is extended to $M^{\Omega}$ by uniform limits. We recall the following quantifier elimination result for $T^R$ (see \cite[Theorem 3.32$(i)$]{BYRV}).

\begin{fact}\label{fact:randomizationsQE}
For any $\bar{h}$ in $\mathcal{C}$, $\tp(\bar{h})$ is uniquely determined by the values of $\Eb[\varphi(\bar{h})]$ for $\cL$-formulas $\varphi(\bar{x})$.  
\end{fact}

For the rest of this section, we use $x$ for a possibly infinite tuple of variables. We write $S_x(-)$ for type spaces in $T$ and $S^R_x(-)$ for type spaces in $T^R$.

Given $M\models T$ and $m\in M$, let $\tilde{m}\in M^\Omega$ be represented by the constant function in $M_0^\Omega$ with value $m$. Let $M^c=\{\tilde{m}:m\in M\}\seq M^\Omega$. The following is a restatement of \cite[Theorem 3.32$(ii)$]{BYRV} suitable for our setting. 

\begin{fact}\label{fact:nup}
Given $M\models T$, any type $p\in S_x^R(M^c)$ determines a unique measure $\nu_p\in \kM_x(M)$ defined so that, for any $\cL_M$-formula $\theta(x)$,
\[
\nu_p(\theta(x))=(\Eb[\theta(x)])^p.
\]
Moreover, the map $p\mapsto\nu_p$ from $S_x^R(M^c)$ to $\kM_x(M)$ is a homeomorphism.
\end{fact}
\begin{proof}[Explanation]
Let $T_M$ be the theory $T$ with constants for $M$. In the randomization $T^R_M$, the constant symbol for an element $m\in M$ is interpreted as $\tilde{m}$, which provides a copy of $M^c$ in any model of $T^R_M$. Applying  \cite[Theorem 3.32$(ii)$]{BYRV} to $T_M$ and $T_M^R$ yields the statement.
\end{proof}

The previous fact allows us to extend the $\Eb$ map to arbitrary continuous functions. In particular, given $M\models T$ and a continuous function $F\colon S_x(M)\to\R$, we have a continuous function $\Eb[F]\colon S^R_x(M^\Omega)\to \R$ such that, given $p\in S^R_x(M^\Omega)$,
\[
\Eb[F](p)=\int_{S_x(M)}F\,d\nu_{p|_{M^c}}.
\]
Note that $\Eb[F]$ is continuous since the maps $p \mapsto \nu_{p|_{M^c}}$ and $\nu \mapsto \int F d\nu$ are continuous (where $\mathfrak{M}_{x}(M)$ is equipped with the weak$^*$-topology).

Fact \ref{fact:nup} also motivates the following terminology.

\begin{definition}\label{def:extend}
Fix $M\models T$ and $\mu\in\kM_x(M)$. Given $\mathcal{N}\succeq M^\Omega$ and $p\in S^R_x(\mathcal{N})$, we say that $p$ \textbf{extends} $\mu$ if $\nu_{p|_{M^c}}=\mu$, i.e., for any $\cL_M$-formula $\theta(x)$,
\[
\mu(\theta(x))=(\Eb[\theta(x)])^p.
\]
\end{definition}

In the context of the previous definition, note that we may view $\mu$ as a type in $S^R_x(M^c)$ (via Fact \ref{fact:nup}), and  $p$ extends $\mu$ if and only if it extends (in the usual sense) this type associated to $\mu$. This terminology has precedents in  \cite{BYT,Gannon-note,Khanaki}.\footnote{In \cite{Khanaki}, Khanaki uses ``randomly extends"; we drop ``randomly" for the sake of brevity.} 

Finally, we recall Ben Yaacov's transfer map from \cite{BYT}. As in \cite{CGH2}, we let $\cU$ denote a monster model of $T$, and we let $\mathcal{C}$ denote a monster model of $T^R$ with $\mathcal{U}^\Omega\preceq \mathcal{C}$. 

\begin{fact}\label{fact:rtrans}
Suppose $\mu\in\kM_x(\cU)$ is definable over $M\prec\cU$. Then there is a unique $M^\Omega$-definable type $r_\mu\in S^R_x(\mathcal{C})$ such that for any $\cL_M$-formula $\varphi(x,y)$, 
\begin{equation*}
F^{\Eb[\varphi(x,y)]}_{r_\mu}=\Eb[F^{\varphi(x,y)}_\mu].\tag{$\ast$}
\end{equation*}
In particular, $r_\mu$ extends $\mu$. 
 Moreover, for any $n\geq 1$, we have $r^{(n)}_\mu=r_{\mu^{(n)}}$. 
\end{fact}

A detailed proof of this fact is given in \cite{CGH2} for discrete $T$. Only minimal modifications are required for continuous $T$, which we describe in Subsection \ref{sec:randoproofs}. Note that condition $(\ast)$ is generally  much stronger  than simply saying $r_\mu$ extends $\mu$, which is why we previously referred to $r_\mu$ as a ``canonical extension'' of $\mu$.

\section{From generic stability of $r_\mu$ to \emph{fim} for $\mu$}\label{sec:rgsfim}

We first recall the definition of generic stability for types.

\begin{definition}\label{def:gs}
A global type $p\in S_x(\cU)$ is \textbf{generically stable over $M\prec\cU$} if $p$ is $M$-invariant and, for any Morley sequence $(a_i)_{i<\omega}$ in $p$ over $M$, we have
\[
\lim_{i\to\infty}\tp(a_i/\cU)=p.
\]
\end{definition}

See \cite[Fact 3.4]{CGH2} for various equivalent formulations of generic stability in continuous theories, which are based on previous results in discrete theories from \cite{PiTa} and \cite{CoGa}. We also note that even though the previous definition is formulated in $T$, we can (and will) apply it to global types in $T^R$ over the monster $\Cc$.

 Given a definable measure $\mu\in\kM_x(\cU)$, an $\cL$-formula $\varphi(x,y)$, and an integer $n\geq 1$, define
 \[
 D^\varphi_{\mu,n}(\xbar,y)= \Av(\xbar)(\varphi(x,y))-F^\varphi_\mu(y),
 \]
 where $\xbar=(x_1,\ldots,x_n)$. We view $D^\varphi_{\mu,n}(\xbar,y)$ as a continuous function from $S_{\xbar y}(\cU)$ to $\R$. We now state a result from \cite{CGH2} using this notation. The proof in \cite{CGH2} is for discrete $T$, and the adaptation to continuous $T$ is explained in Subsection \ref{sec:ABC}.

\begin{proposition}[\cite{CGH2}]\label{prop:ABC}
 Suppose $\mu\in\kM_x(\cU)$ is definable over $M\prec\cU$.
\begin{enumerate}[$(a)$]
\item  $\mu$ is \fim\ over $M$ if and only if for every $\cL$-formula $\varphi(x,y)$, 
  \[ 
  \lim_{n \to \infty} \textstyle \left(\sup_y\mathbb{E}[|D^\varphi_{\mu,n}(\xbar,y)|]\right)^{r^{(n)}_\mu} = 0.
  \]
  \item $r_\mu$ is generically stable over $M^\Omega$ if and only if for every $\cL$-formula $\varphi(x,y)$, 
  \[
  \lim_{n\to\infty}\textstyle \left(\sup_y |\mathbb{E} [D^\varphi_{\mu,n}(\xbar,y)]|\right)^{r^{(n)}_\mu}=0.
  \]
  \end{enumerate}
\end{proposition}
As in \cite{CGH2}, the previous result has the following immediate consequence, which establishes the $(i)\Rightarrow(ii)$ direction of Theorem \ref{thm:main}.

\begin{corollary}[\cite{CGH2}]\label{cor:fimtorgs}
Suppose $\mu\in\kM_x(\cU)$ is fim over $M\prec\cU$. Then $\mu$ is definable over $M$ and  $r_\mu$ is generically stable over $M^\Omega$.
\end{corollary}
\begin{proof}
Definability of \fim\ measures is a standard result (see \cite[Fact 3.4$(a)$]{CGH2} and \cite[Lemma 3.2]{AndGS}). Hence, generic stability of $r_\mu$ follows directly from Proposition \ref{prop:ABC} and Jensen's inequality.
\end{proof}

 While it is less obvious that Proposition \ref{prop:ABC}  also implies the $(ii)\Rightarrow(i)$ direction of Theorem \ref{thm:main}, it turns out that this is indeed the case due to an elementary coding trick. 

\begin{theorem}\label{thm:rgstofim}
Suppose $\mu\in\kM_x(\cU)$ is definable over $M\prec\cU$ and $r_\mu$ is generically stable over $M^\Omega$. Then $\mu$ is $\fim$ over $M$.
\end{theorem}
\begin{proof}
To ease notation, let $r_n=r^{(n)}_\mu|_{M^\Omega}$. 
By Proposition \ref{prop:ABC}$(a)$, it suffices to fix an $\cL$-formula $\varphi(x,y)$ and prove
\begin{equation}\label{eq:goal}
\lim_{n\to\infty}\left(\sup_y\Eb[|D^\varphi_{\mu,n}(\xbar,y)|]\right)^{r_n}=0.
\end{equation}
 Without loss of generality, we may assume $\varphi(x,y)$ is $[0,1]$-valued. Fix some non-constant $\cL$-formula $\chi(z)$ (in some tuple $z$ of variables). After shifting and scaling, we may assume there are $c_0,c_1\in M^z$ such that $\chi(c_0)=0$ and $\chi(c_1)=1$. Define the formula 
\[
\hat{\varphi}(x;y,z)={\min}\big({\max}(\chi(z),\varphi(x,y)),\max(1-\chi(z),1-\varphi(x,y))\big).
\]
In particular,
\[
\hat{\varphi}(x;y,c_0)=\varphi(x,y)\quad\text{and}\quad\hat{\varphi}(x;y,c_1)=1-\varphi(x,y),
\]
which implies
\begin{equation*}
D^{\hat{\varphi}}_{\mu,n}(\xbar;y,c_0)=D^\varphi_{\mu,n}(\xbar,y)\quad\text{ and }\quad D^{\hat{\varphi}}_{\mu,n}(\xbar;y,c_1)=-D^\varphi_{\mu,n}(\xbar,y).\tag{$\ast$}
\end{equation*}
Given $n\geq 1$,  define the following continuous functions on $S_{x_1,\ldots,x_n}(M^\Omega)$:
\[
G_n(\xbar)={\textstyle\sup_y}\Eb[|D^\varphi_{\mu,n}(\xbar,y)|]\quad\text{and}\quad H_n(\xbar)={\textstyle\sup_{yz}}\left|\Eb[D^{\hat{\varphi}}_{\mu,n}(\xbar;y,z)]\right|.
\]
We emphasize that $G_{n}(\bar{x})$ is constructed using the formula $\varphi$, while $H_{n}(\bar{x})$ uses $\hat{\varphi}$. Note that our goal (\ref{eq:goal}) is $\lim_{n\to\infty} G_n(r_n)=0$. On the other hand, since $r_\mu$ is generically stable over $M^\Omega$, it follows from Proposition \ref{prop:ABC}$(b)$ that 
\[
\lim_{n\to\infty}H_n(r_n)=0.
\]
 Therefore, to establish (\ref{eq:goal}), it suffices to fix $n\geq 1$ and prove the pointwise inequality 
 \begin{equation}\label{eq:ineq1}
 G_n(\xbar)\leq H_n(\xbar).
 \end{equation}
 By  density of $M_0^\Omega$ in $M^\Omega$, it suffices to verify (\ref{eq:ineq1}) on $M_0^\Omega$. So fix $\abar\in (M_0^\Omega)^{\xbar}$ and $b\in (M_0^\Omega)^y$. We will construct $c\in (M_0^\Omega)^z$ such that 
 \begin{equation}\label{eq:ineq2}
 \Eb[|D^\varphi_{\mu,n}(\abar,b)|]=  \left|\Eb[D^{\hat{\varphi}}_{\mu,n}(\abar;b,c)]\right|.
 \end{equation}
 Consider the set
 \[
 A=\{t\in\Omega:D^\varphi_{\mu,n}(\abar(t),b(t))\geq 0\}.
 \]
Since $D^\varphi_{\mu,n}(\xbar,y)$ is a definable predicate in $T$, we have $A\in\cB$. Thus we may define $c\in (M_0^\Omega)^z$ so that
 \[
 c(t)=\begin{cases}
 c_0 & \text{if $t\in A$,}\\
 c_1 & \text{if $t\not\in A$.}
 \end{cases}
 \]
 Then
\begin{align*}
  \Eb[|D^\varphi_{\mu,n}(\abar,b)|] &= \int_{\Omega} |D^{\varphi}_{\mu,n}(\abar(t),b(t))|\, d\mathbb{P} \\  
  &= \int_{A}  D^{\varphi}_{\mu,n}(\abar(t),b(t))\, d\mathbb{P} - \int_{\Omega \backslash A} D^{\varphi}_{\mu,n}(\abar(t),b(t))\, d\mathbb{P} \\ 
  &= \int_{A}  D^{\hat{\varphi}}_{\mu,n}(\abar(t),b(t),c(t))\, d\mathbb{P} + \int_{\Omega \backslash A} D^{\hat{\varphi}}_{\mu,n}(\abar(t),b(t),c(t))\, d\mathbb{P} \\ 
  &= \int_{\Omega} D^{\hat{\varphi}}_{\mu,n}(\abar(t),b(t),c(t))\, d\mathbb{P} \\
  &= \Eb[D^{\hat{\varphi}}_{\mu,n}(\abar;b,c)]\\
  &= \left|\Eb[D^{\hat{\varphi}}_{\mu,n}(\abar;b,c)]\right|, 
\end{align*}
where the third equality uses $(\ast)$. This establishes (\ref{eq:ineq2}). 
\end{proof}

\section{From self-averaging for $\mu$ to generic stability of $r_\mu$}\label{sec:sargs}

 Given a variable $x$, we will use $\bfx$ for an infinite sequence $(x_i)_{i<\omega}$ of variables, each of the same sort as $x$.  The following notion of self-averaging for measures was introduced in \cite{CGH2}.

\begin{definition}\label{def:SA}
Suppose $\mu\in\kM_x(\cU)$ is Borel-definable over $M\prec\cU$.
\begin{enumerate}[$(1)$]
\item Define $\KM(\mu/M)=\{\lambda\in\kM_{\bfx}(\cU):\lambda|_M=\mu^{(\omega)}|_M\}$. Measures in $\KM(\mu/M)$ are called \textbf{Keisler-Morley measures in $\mu$ over $M$}. 
\item $\mu$ is \textbf{self-averaging over $M$} if for any $\lambda\in\KM(\mu/M)$,
\[
\lim_{i\to\infty}\lambda|_{x_i}=\mu.
\]
\end{enumerate}
\end{definition}

If one views Keisler-Morley measures as analogues of Morley sequences for types, then self-averaging is analogous to generic stability as defined in Definition \ref{def:gs}. 

In the case that $T$ is discrete, Remark 3.21 of \cite{CGH2} characterizes self-averaging for $\mu$ as generic stability of $r_\mu$ with respect to formulas over constant parameters in $\cU^c$ (see Subsection \ref{sec:constants} for details). Consequently, generic stability for $r_\mu$ implies self-averaging for $\mu$. In this section, we will recover this result for continuous $T$ and prove the converse. The proof will consist of two main steps:

\begin{enumerate}[$(1)$]
\item In Proposition \ref{prop:SSA}, we characterize generic stability of $r_\mu$ by a  strengthening of the definition of self-averaging designed to address the discrepancy in the parameters mentioned above. For discrete $T$, this  characterization is largely evident from work of Chernikov, Gannon, and Krupi\'{n}ski  \cite{CGK} (see Remark \ref{rem:SSA}). 
\item We  prove that self-averaging implies the stronger form captured in the previous step. This is the portion of the proof requiring substantial new techniques. 
\end{enumerate}

In order to state Proposition \ref{prop:SSA}, we first modify the notion of a Keisler-Morley measure. In particular, rather than considering global measures in $\bfx$, we will focus on measures over $M$ in $\bfx$ together with additional variables $y$, which in some sense will function as ``random parameters" in the arguments below.

 \begin{definition}
  Suppose $\mu\in\kM_x(\cU)$ is Borel-definable over $M$. Given a tuple $y$ of variables, define
\[
\KM_y(\mu/M)=\{\sigma\in\kM_{\bfx y}(M):\sigma|_{\bfx}=\mu^{(\omega)}|_M\}. 
\]
  \end{definition}

In the context of the previous definition, measures in $\KM(\mu/M)$ can be associated to measures in $\KM_y(\mu/M)$ satisfying the additional property that the restriction to $y$ is a type (see Lemma \ref{lem:weird-extension} below for a more precise statement).

We will also use the following lemma.

\begin{lemma}\label{lem:measure-extension}
Fix $M\models T$, $p\in S^R_{\bfx}(M^\Omega)$, and $\sigma\in\kM_{\bfx y}(M)$ with $p$ extending $\sigma|_{\bfx}$. Then there is a type $q\in S^R_{\bfx y}(M^\Omega)$ such that $q|_{\bfx}=p$ and $q$ extends $\sigma$. 
\end{lemma}
\begin{proof}
By Fact \ref{fact:nup}, there is a type $r\in S^R_{\bfx y}(M^c)$ extending $\sigma$. Choose $\boldsymbol{a}\in \mathcal{C}^{\bfx}$ and $\boldsymbol{a}'b'\in \mathcal{C}^{\bfx y}$ such that $\boldsymbol{a}\models p$ and $\boldsymbol{a}'b'\models r$. Note that for any $\cL_M$-formula $\theta(\bfx)$, we have 
\[
\Eb[\theta(\boldsymbol{a})]=(\Eb[\theta(\bfx)])^p=\sigma(\theta(\bfx))=(\Eb[\theta(\bfx)])^r=\Eb[\theta(\boldsymbol{a}')].
\]
Therefore, by quantifier elimination, we have $\boldsymbol{a}\equiv_{M^c} \boldsymbol{a}'$. Choose some $b\in\mathcal{C}^y$ such that $\boldsymbol{a}b\equiv_{M^c}\boldsymbol{a}'b'$, and let $q=\tp(\boldsymbol{a}b/M^\Omega)$. Then $q|_{\bfx}=p$ and hence  $q$ extends $\sigma$ since $q|_{M^c}=r$. 
\end{proof}

We can now state and prove Proposition \ref{prop:SSA}.
Given a global measure $\mu\in\kM_x(\cU)$, which is Borel-definable over $M$, and some measure $\nu\in\kM_y(M)$, we will abuse notation and write $\mu\otimes \nu$ for the resulting restricted Morley product in $\kM_{xy}(M)$.

\begin{proposition}\label{prop:SSA}
Suppose $\mu\in\kM_x(\cU)$ is definable over $M\prec\cU$. Then the following are equivalent. 
\begin{enumerate}[$(i)$]
\item $r_\mu$ is generically stable over $M^{\Omega}$.
\item For any variable tuple $y$, if $\sigma\in \KM_y(\mu/M)$ then $\lim_{i\to\infty}\sigma|_{x_i y}=\mu\otimes\sigma|_y$.
\end{enumerate}
\end{proposition}
\begin{proof}
$(i)\Rightarrow(ii)$.
Assume $r_\mu$ is generically stable over $M^\Omega$ and fix $\sigma\in\KM_y(\mu/M)$. Let $p=r_\mu^{(\omega)}|_{M^\Omega}$. Then for any $n>0$, we have 
\[
p|_{x<n}=r^{(n)}_\mu|_{M^\Omega}=r_{\mu^{(n)}}|_{M^\Omega},
\]
where the final equality is by Fact \ref{fact:rtrans}. Since $r_{\mu^{(n)}}$ extends $\mu^{(n)}$ (by Fact \ref{fact:rtrans}) and $\sigma|_{\bfx}=\mu^{(\omega)}|_M$, it follows that $p$ extends $\sigma|_{\bfx}$. By Lemma \ref{lem:measure-extension}, there is $q\in S^R_{\bfx y}(M^\Omega)$ such that $q|_{\bfx}=p$ and $q$ extends $\sigma$. Let $(a_i)_{i<\omega}b\models q$. Then  $(a_i)_{i<\omega}$ is a Morley sequence in $r_\mu$ over $M^\Omega$. Altogether, given an $\cL_M$-formula $\varphi(x,y)$, we have
 \begin{multline*}
 \lim_{i\to\infty}\sigma(\varphi(x_i,y))\overset{(a)}{=}\lim_{i\to\infty}\Eb[\varphi(a_i,b)]\overset{(b)}{=}(\Eb[\varphi(x,b)])^{r_\mu}\\
 \overset{(c)}{=}\Eb[F^{\varphi(x,y)}_\mu](\tp(b/M^\Omega))
 \overset{(d)}{=}\int_{S_y(M)} F^{\varphi(x,y)}_\mu\,d\sigma|_y=(\mu\otimes\sigma|_y)(\varphi(x,y)),
 \end{multline*}
 where the labeled equations have the following justifications:
 \begin{enumerate}[$(a)$]
 \item for each $i<\omega$, $\sigma(\varphi(x_i,y))=(\Eb[\varphi(x_i,y)])^q=\Eb[\varphi(a_i,b)]$;
 \item by $(i)$ and since $(a_i)_{i<\omega}$ is a Morley sequence in $r_\mu$ over $M^\Omega$; 
 \item by Fact \ref{fact:rtrans};
 \item since $\tp(b/M^\Omega)$ extends $\sigma|_y$.
 \end{enumerate}
 Since $\varphi(x,y)$ was arbitrary, this establishes $\lim_{i\to\infty}\sigma|_{x_i y}=\mu\otimes\sigma|_y$.\medskip
 
$(ii)\Rightarrow (i)$.  Assume $(ii)$. To show $r_\mu$ is generically stable over $M^\Omega$, we fix a Morley sequence $(a_i)_{i<\omega}$ in $r_\mu$ over $M^\Omega$, and show that $\lim_{i\to\infty}\tp(a_i/\mathcal{C})=r_\mu$.
By  Fact \ref{fact:randomizationsQE}, it suffices to fix an $\cL$-formula $\varphi(x,y)$ and some $b\in\mathcal{C}^y$, and prove
\[
\lim_{i\to\infty}\Eb[\varphi(a_i,b)]=(\Eb[\varphi(x,b)])^{r_\mu}.
\]

 Let $p=\tp((a_i)_{i<\omega}b/M^\Omega)\in S^R_{\bfx y}(M^\Omega)$. Let $\sigma=\nu_{p|_{M^c}}\in\kM_{\bfx y}(M)$ (recall Fact \ref{fact:nup}). So $p$ extends $\sigma$. We claim that $\sigma|_{\bfx}=\mu^{(\omega)}|_M$, and hence $\sigma\in\KM_y(\mu/M)$. Indeed, as before, Fact \ref{fact:rtrans} implies that for any $n>0$, $p|_{x_{<n}}=r_{\mu^{(n)}}|_{M^\Omega}$, and hence $p|_{x_{<n}}$ extends both $\sigma|_{x_{<n}}$ and $\mu^{(n)}|_M$. Thus  these measures are the same.

Now we have
 \begin{multline*}
 \lim_{i\to\infty}\Eb[\varphi(a_i,b)] \overset{(a)}{=} \lim_{i\to\infty}\sigma(\varphi(x_i,y)) \overset{(b)}{=} (\mu\otimes\sigma|_y)(\varphi(x,y))\\
 =\int_{S_y(M)}F^{\varphi(x,y)}_\mu\,d\sigma|_y \overset{(c)}{=} \Eb[F^{\varphi(x,y)}_\mu](\tp(b/M^\Omega)) \overset{(d)}{=} (\Eb[\varphi(x,b)])^{r_\mu},
 \end{multline*}
  where the labeled equations have the following justifications:
 \begin{enumerate}[$(a)$]
 \item for each $i<\omega$, $\Eb[\varphi(a_i,b)]=(\Eb[\varphi(x_i,y)])^p=\sigma(\varphi(x_i,y))$ ;
 \item by $(ii)$ and since $\sigma\in\KM_y(\mu/M)$;
  \item since $\tp(b/M^\Omega)$ extends $\sigma|_y$;
 \item by Fact \ref{fact:rtrans}.\qedhere
 \end{enumerate}
\end{proof}

\begin{corollary}\label{cor:rgstoSA}
Suppose $\mu\in\kM_x(\cU)$ is definable over $M\prec\cU$ and $r_\mu$ is generically stable over $M^\Omega$. Then $\mu$ is self-averaging over $M$. 
\end{corollary}
\begin{proof}
More specifically, we observe that if $\mu\in\kM_x(\cU)$ is Borel-definable over $M\prec\cU$ and satisfies condition $(ii)$ of Proposition \ref{prop:SSA}, then $\mu$ is self-averaging over $M$. Indeed, assuming condition $(ii)$, fix some $\lambda\in \KM(\mu/M)$. Fix an $\cL_{\cU}$-formula $\varphi(x,b)$ with $b\in\cU^y$. Define $\sigma\in\kM_{\bfx y}(M)$ so that $\sigma(\theta(\bfx,y))=\lambda(\theta(\bfx,b))$ for any $\cL_M$-formula $\theta(\bfx,y)$. Then $\sigma\in\KM_y(\mu/M)$ and thus, by condition $(ii)$, we have
\[
 \lim_{i\to\infty}\lambda(\varphi(x_i,b))=\lim_{i\to\infty}\sigma(\varphi(x_i,y))=(\mu\otimes\sigma|_y)(\varphi(x,y))=\mu(\varphi(x,b)).\qedhere
 \]
\end{proof}

\begin{remark}\label{rem:SSA}
Proposition \ref{prop:SSA} and the proof of Corollary \ref{cor:rgstoSA} bear a strong resemblance to parts of \cite[Section 3]{CGK}. In particular, \cite[Proposition 3.23]{CGK} establishes the $(i)\Rightarrow (ii)$ direction of Proposition \ref{prop:SSA} when $T$ is discrete. Moreover, the proof is nearly the same, except that we replace the use of Ben Yaacov's general ``natural extension" construction (see \cite[Fact 3.17]{CGK}) with Lemma \ref{lem:measure-extension}. We also note that condition $(ii)$ of Proposition \ref{prop:SSA} is formulated in \cite[Proposition 3.23]{CGK} using measures $\sigma\in\kM_{\bfx y}(\cU)$ with $\sigma|_{\bfx,M}=\mu^{(\omega)}|_M$. This can be derived from our formulation by adding more variables for parameters from $\cU$ (similar to the proof of Corollary \ref{cor:rgstoSA}). Finally,  \cite[Theorem 3.13]{CGK} involves a strengthening of condition $(ii)$ with additional uniformity. This is done via a compactness argument based on \cite[Lemma 2.3]{CGH2} (cf. \cite[Proposition 2.5]{CGH2}), which works also in the continuous setting. 
\end{remark}

We now start toward  the converse of Corollary \ref{cor:rgstoSA}.  We will need the following result,  which is proved  in \cite{CGH2} when $T$ is discrete. The adaptation to continuous $T$ requires only minimal adjustments, which we describe in the appendix. 

\begin{lemma}[\cite{CGH2}]\label{lem:SAfam}
Suppose $\mu\in\kM_x(\cU)$ is self-averaging over $M\prec\cU$. Then $\mu$ is finitely approximable over $M$. In particular, $\mu$ is definable over $M$ and $\mu$ is self-commuting (i.e., $\mu_x\otimes\mu_{x'}=\mu_{x'}\otimes\mu_x$).
\end{lemma}

We will also use the following continuous analogue of  \cite[Lemma 2.9]{CGH2}. 

\begin{lemma}\label{lem:weird-extension}
 Suppose $\mu\in\kM_x(\cU)$ is Borel-definable over $M\prec\cU$ and   $\sigma \in \KM_y(\mu/M)$ is such that  $\sigma |_y$ is a type in $S_y(M)$. Fix $b\in\Uc^y$ realizing $\sigma|_y$. Then there is  some $\lambda \in \KM(\mu/M)$ such that, for any $\cL_M$-formula $\psi(\bfx,y)$, $\lambda(\psi(\bfx,b)) = \sigma(\psi(\bfx,y))$.
\end{lemma}
\begin{proof}
The proof is essentially identical to \cite{CGH2}. By a straightforward exercise, there is a well-defined measure $\lambda_0 \in \frk{M}_{\bfx}(Mb)$ such that $\lambda_0(\psi(\bfx,b)) = \sigma(\psi(\bfx,y))$ for any $\cL_M$-formula $\psi(\bfx,y)$. Now let $\lambda\in\kM_{\bfx}(\cU)$ be a global extension (which exists by standard results in functional analysis, e.g., \cite[Section 1]{DitEif}).
\end{proof}

Given $n\geq 1$, let $\Sym(n)$ denote the group of permutations of $\omega$ that fix every point $i\geq n$. Let $\fSym(\omega)=\bigcup_{n\geq 1}\Sym(n)$. Now suppose $y$ is a (possibly empty) tuple of variables. Given $\sigma\in\kM_{\bfx y}(M)$ and $\gamma\in \fSym(\omega)$, let $\gamma_*\sigma$ denote the measure in $\kM_{\bfx y}(M)$ obtained as the pushforward of $\sigma$ along the map from $S_{\bfx y}(M)$ to $S_{\bfx y}(M)$ which permutes the variables in $\bfx$ according to $\gamma$. We say that $\sigma$ is \textbf{$\gamma$-invariant} if $\gamma_*\sigma=\sigma$, and $\sigma$ is \textbf{$\fSym(\omega)$-invariant} if it is $\gamma$-invariant for all $\gamma\in\fSym(\omega)$. Define 
\[
\kM^{\sym}_{\bfx y}(M)=\{\sigma\in\kM_{\bfx y}(M):\sigma\text{ is $\fSym(\omega)$-invariant}\}.
\]
Note that $\kM^{\sym}_{\bfx y}(M)$ is a closed convex set in $\kM_{\bfx y}(M)$. 

The next lemma is an adaptation of a classical result of Hewitt and Savage \cite{HewSav}. See Remark \ref{rem:HS} below for further details. 

\begin{lemma}\label{lem:ergodic}
Suppose $\mu\in\kM_x(\cU)$ is definable over $M\prec\cU$ and self-commuting. Then $\mu^{(\omega)}|_M$ is an extreme point in $\kM^{\sym}_{\bfx}(M)$. 
\end{lemma}
\begin{proof}
Given $\lambda\in\kM_{\bfx}(M)$ and $f\in C(S_{\bfx}(M),\R)$, to simplify notation we will write $\lambda(f)$ for $\int_{S_{\bfx}(M)}f\,d\lambda$. 

Set $\pi=\mu^{(\omega)}|_M$. Since $\mu$ is self-commuting, it follows that $\pi$ is in $\kM_{\bfx}^{\sym}(M)$. Suppose we can write $\pi=t\lambda+(1-t)\lambda'$ for some $\lambda,\lambda'\in\kM^{\sym}_{\bfx}(M)$ and $t\in (0,1)$. We want to show $\lambda=\pi$.  By the Riesz representation theorem, it suffices to show that $\lambda(f)=\pi(f)$ for any $f\in C(S_{\bfx}(M),\R)$.

Given $I\seq \omega$, we write $S_I(M)$ for $S_{x_I}(M)$ and, if $p\in S_{\bfx}(M)$ then we write $p|_I$ for $p|_{x_I}$. Given $f\in C(S_{\bfx}(M),\R)$ and $I\seq\omega$, we say that $f$ is \emph{$I$-invariant} if, for any $p,q\in S_{\bfx}(M)$, we have $f(p)=f(q)$ whenever $p|_I=q|_I$. Let $\mathcal{A}$ be the set of $f\in C(S_{\bfx}(M),\R)$ that are $I$-invariant for some finite $I\seq\omega$. It is straightforward to show that $\mathcal{A}$ is a unital subalgebra of $C(S_{\bfx}(M),\R)$ that separates points in $S_{\bfx}(M)$. Therefore, by Stone-Weierstrass and the fact that integration commutes with uniform limits, it suffices to show $\lambda(f)=\pi(f)$ for all $f\in\mathcal{A}$. 

Fix $f\in\mathcal{A}$ and fix a finite set $I\seq\omega$ such that $f$ is $I$-invariant. Let $n\geq 1$ be arbitrary and fix pairwise disjoint sets $I_1,\ldots,I_n\seq\omega$ each of size $|I|$. For $t\in[n]$, let $\alpha_t\colon S_{I_t}(M)\to S_{I}(M)$ be a variable substitution map defined from some fixed bijection between $I_t$ and $I$. Define $f_t\in C(S_{\bfx}(M),\R)$ so that, given $p\in S_{\bfx}(M)$, $f_t(p)=f(q)$ where $q\in S_{\bfx}(M)$ satisfies $q|_I=\alpha_t(p|_{I_t})$ (note that $f(q)$ is well-defined by $I$-invariance). Note that each $f_t$ is a finite-coordinate permutation of $f$.

Set $a=\pi(f)$ and $b=\pi(f^2)$. Define  $h=\sum_{t=1}^n (f_t-a)$.\medskip

\noindent\textit{Claim.} $\pi(h^2)=n(b-a^2)$.

\noindent\textit{Proof.} Since $\pi$ is  $\fSym(\omega)$-invariant, we have $\pi(f_t)=a$ and $\pi(f_t^2)=b$ for all $t\in [n]$. Recall $\pi=\mu^{(\omega)}|_M$. Together with   $\fSym(\omega)$-invariance, one sees that for any distinct $s,t\in [n]$,  we have
\[
\pi(f_sf_t)=(\mu^{(I_s)}\otimes\mu^{(I_t)})(f_s f_t)=\mu^{(I_s)}(f_s)\mu^{(I_t)}(f_t)=\pi(f_s)\pi(f_t).
\]
So by linearity (and since $\pi$ is a probability measure),
\[
\pi((f_s-a)(f_t-a))=\pi(f_s f_t)-a\pi(f_s)-a\pi(f_t)+a^2=a^2-2a^2+a^2=0.
\]
This implies
\[
\pi(h^2) = \sum_{t=1}^n\pi((f_t-a)^2)=\sum_{t=1}^n\big(\pi(f_t^2)-2a\pi(f_t)+a^2\big)=n(b-a^2).\clqed
\]

Now, since $\lambda$ is $\fSym(\omega)$-invariant, we have $\lambda(h)=n(\lambda(f)-a)$ by linearity. Therefore
\begin{multline*}
|\lambda(f)-\pi(f)|=|\lambda(f)-a|=\frac{|\lambda(h)|}{n}\leq \left(\frac{\lambda(h^2)}{n^2}\right)^{1/2}\\
\leq \left(\frac{t\inv\pi(h^2)}{n^2}\right)^{1/2}=\left(\frac{t\inv (b-a^2)}{n}\right)^{1/2},
\end{multline*}
where the first inequality uses Cauchy-Schwarz (recall that $\lambda$ is a probability measure), the second inequality uses $\pi=t\lambda+(1-t)\lambda'$ and the fact that $h^2$ is nonnegative on $S_{\bfx}(M)$, and the final equality uses the claim.
Since $n$ was arbitrary, this yields $\lambda(f)=\pi(f)$, as desired.
\end{proof}

\begin{remark}
Note that in the previous lemma, the assumption that $\mu$ is definable over $M$ can be weakened to the assumption that all Morley powers $\mu^{(n)}$ are Borel-definable over $M$, and the Morley product is associative for such powers. In \cite{GHevents}, this condition is referred to as ``$M$-adequacy", and ``$M$-excellence" is defined to be $M$-adequacy plus self-commuting. Thus the proof of Lemma \ref{lem:ergodic} shows that if $\mu\in\kM_x(\cU)$ is $M$-excellent, then $\mu^{(\omega)}|_M$ is an extreme point in $\kM^{\sym}_{\bfx}(M)$. 
\end{remark}

\begin{remark}\label{rem:HS}
The aforementioned work of Hewitt and Savage establishes a direct analogue of Lemma \ref{lem:ergodic} for symmetric product measures on infinite Cartesian powers of a measurable space $X$ (see Theorems 5.1 and 5.2 of \cite{HewSav}). We cannot apply these results directly in our context since $S_{\bfx}(\cU)$ is not a genuine product space. However, the proof of Lemma \ref{lem:ergodic} follows essentially the same second-moment argument from this earlier work. Similar results appear in \cite[Theorem 2.3]{Hoover} and \cite[Section 6]{Matus}. 
\end{remark}

We can now prove the main result of this section.

\begin{theorem}\label{thm:SAtorgs}
Suppose $\mu\in\kM_x(\cU)$ is self-averaging over $M\prec\cU$. Then $r_\mu$ is generically stable over $M^\Omega$.
\end{theorem}
\begin{proof}
First, by Lemma \ref{lem:SAfam}, $\mu$ is definable over $M\prec\cU$, and hence $r_\mu$ exists. Also, $\mu$ is self-commuting. We will prove generic stability of $r_\mu$ using condition $(ii)$ of Proposition \ref{prop:SSA}. Toward this end, fix a tuple $y$ of variables. Define
\[
\mathcal{K}=\kM^{\sym}_{\bfx y}(M)\cap \KM_y(\mu/M).
\]
 Note that $\mathcal{K}$ is a closed convex set in $\kM_{\bfx y}(M)$.\medskip

\noindent\textit{Claim 1.} Suppose $\sigma$ is an extreme point in $\mathcal{K}$. Then $\sigma|_y$ is a type in $S_y(M)$.

\noindent\textit{Proof.} 
This follows from Lemma \ref{lem:ergodic} using standard facts and basic exercises in convexity theory. For the sake of clarity, we include some explanation. First, we claim that $\sigma$ is also an extreme point in $\kM^{\sym}_{\bfx y}(M)$ (cf. \cite[Lemma 2.1]{BRZK}). Indeed,  suppose $\sigma=t\sigma_1+(1-t)\sigma_2$, where $t\in (0,1)$ and $\sigma_1,\sigma_2\in\kM^{\sym}_{\bfx y}(M)$. Restricting to $\bfx$, we have $\mu^{(\omega)}|_M=t\sigma_1|_{\bfx}+(1-t)\sigma_2|_{\bfx}$. So by Lemma \ref{lem:ergodic}, $\mu^{(\omega)}|_M=\sigma_1|_{\bfx}=\sigma_2|_{\bfx}$. Therefore $\sigma_1,\sigma_2\in \mathcal{K}$, and hence $\sigma=\sigma_1=\sigma_2$ since $\sigma$ is extreme in $\mathcal{K}$. 

Now, since $\sigma$ is an extreme point in $\kM^{\sym}_{\bfx y}(M)$, it follows that $\sigma$ is ergodic with respect to the ambient action of $\fSym(\omega)$ on $S_{\bfx y}(M)$ (see, e.g., \cite[Proposition 12.4]{phelps-book}). To explain how this yields the claim, fix a Borel set $B\seq S_y(M)$ and set $B'=\{p\in S_{\bfx y}(M):p|_y\in B\}$. Then  $B'$ is a Borel set in $S_{\bfx y}(M)$ and $\sigma(B')=\sigma|_y(B)$. Since $B'$ is $\fSym(\omega)$-invariant,  ergodicity of $\sigma$ implies $\sigma|_y(B)$ is either $0$ or $1$. Therefore $\sigma|_y$ is a type in $S_y(M)$.\clqed\medskip

Now, given an $\cL_M$-formula $\varphi(x,y)$, define $\ell_\varphi\colon \KM_y(\mu/M)\to\R$ so that 
\[
\ell_\varphi(\sigma)= \sigma(\varphi(x_0,y))-(\mu\otimes\sigma|_y)(\varphi(x,y)).
\]
In other words, $\ell_\varphi(\sigma)=\sigma(\varphi(x_0,y)-F^\varphi_\mu(y))$. Since $\mu$ is definable, it follows that $\ell_\varphi$ is continuous. Moreover, $\ell_\varphi$ is an affine function  in the sense of \cite[p. 13]{phelps-book}. 

\medskip

\noindent\textit{Claim 2.} If $\sigma\in\mathcal{K}$, then $\ell_\varphi(\sigma)=0$ for any $\cL_M$-formula $\varphi(x,y)$.

\noindent\textit{Proof.} Fix an $\cL_M$-formula $\varphi(x,y)$. We want to show that $\ell_\varphi$ is identically $0$ on $\mathcal{K}$. By \cite[Proposition 16.6]{phelps-book} (applied to $\ell_\varphi$ and $-\ell_\varphi$), it suffices to show that $\ell_\varphi(\sigma)=0$ for any extreme point $\sigma$ of $\mathcal{K}$. 

Fix an extreme point $\sigma$ of $\mathcal{K}$. By Claim 1, $\sigma|_y$ is a type in $S_y(M)$. Fix $b\in\cU^y$ realizing $\sigma|_y$ and, using Lemma \ref{lem:weird-extension}, fix $\lambda\in\KM(\mu/M)$  such that, for any $\cL_M$-formula $\theta(\bfx,y)$, $\lambda(\theta(\bfx,b))=\sigma(\theta(\bfx,y))$. Since $\sigma$ is $\fSym(\omega)$-invariant, we have that for all $i<\omega$,
\[
\lambda(\varphi(x_i,b))=\sigma(\varphi(x_i,y))=\sigma(\varphi(x_0,y))=\lambda(\varphi(x_0,b)).
\]
Therefore, since $\mu$ is self-averaging over $M$, we have
\[
\mu(\varphi(x,b))=\lim_{i\to\infty}\lambda(\varphi(x_i,b))=\lambda(\varphi(x_0,b)).
\]
 Altogether,
\[
\sigma(\varphi(x_0,y))=\lambda(\varphi(x_0,b))=\mu(\varphi(x,b))=(\mu\otimes\sigma|_y)(\varphi(x,y)),
\]
as desired.\clqed\medskip

Now, toward a contradiction, suppose there is  $\sigma\in\KM_y(\mu/M)$ and an $\cL_M$-formula $\varphi(x,y)$ such that $\lim_{i\to\infty}\sigma(\varphi(x_i,y))\neq(\mu\otimes\sigma|_y)(\varphi(x,y))$. After passing to a subsequence and relabeling, we may assume that for some fixed $\epsilon>0$, we have $|\sigma(\varphi(x_i,y))-(\mu\otimes\sigma|_y)(\varphi(x,y))|\geq\epsilon$ for all $i<\omega$ (note that $\KM_y(\mu/M)$ is closed under substitution by  infinite subsequences of $\bfx$). After passing to a further subsequence, and replacing $\varphi(x,y)$ with $-\varphi(x,y)$ if necessary, we may assume that for all $i<\omega$,
\begin{equation*}
\sigma(\varphi(x_i,y))-(\mu\otimes\sigma|_y)(\varphi(x,y))\geq\epsilon.\tag{$\ast$}
\end{equation*}

For each integer $n\geq 1$, define 
\[
\sigma_n=\frac{1}{n!}\sum_{\gamma\in\Sym(n)}\gamma_*\sigma.
\]
Since $\mu$ is self-commuting and $\sigma|_{\bfx}=\mu^{(\omega)}|_M$, it follows that $(\gamma_*\sigma)|_{\bfx}=\mu^{(\omega)}|_M$ for any $\gamma\in\fSym(\omega)$, and hence $\sigma_n|_{\bfx}=\mu^{(\omega)}|_M$ for all $n\geq 1$. Thus $\sigma_n\in\KM_y(\mu/M)$ for all $n\geq 1$. Moreover, given $\gamma\in\fSym(\omega)$, since $\gamma_*\sigma$ only differs from $\sigma$ by a permutation of   $\bfx$, we have $(\gamma_*\sigma)|_y=\sigma|_y$. Therefore, given $n\geq 1$, we have
\begin{align*}
\ell_\varphi(\sigma_n) &= \frac{1}{n!}\sum_{\gamma\in\Sym(n)}\big((\gamma_*\sigma)(\varphi(x_0,y))-(\mu\otimes(\gamma_*\sigma)|_y)(\varphi(x,y))\big)\\
 &= \frac{1}{n!}\sum_{\gamma\in\Sym(n)}\big(\sigma(\varphi(x_{\gamma(0)},y))-(\mu\otimes\sigma|_y)(\varphi(x,y))\big)\\
 &\geq \epsilon,
\end{align*}
where the final inequality follows from $(\ast)$. Let $\sigma'\in\KM_y(\mu/M)$ be a limit point of the set $\{\sigma_n:n\geq 1\}$. By continuity, we have $\ell_\varphi(\sigma')\geq\e$. Given a fixed $\gamma\in\fSym(\omega)$,   $\sigma_n$ is $\gamma$-invariant for sufficiently large $n$ 
and thus, since the set of $\gamma$-invariant measures is closed, it follows that $\sigma'$ is $\gamma$-invariant. Therefore $\sigma'\in\mathcal{K}$. But then $\ell_\varphi(\sigma')=0$ by Claim 2, which is a contradiction. 
\end{proof}

\begin{remark}\label{rem:fixedpoint}
The final averaging construction at the end of the previous proof is similar to a ``joining argument" of Ryzhikov \cite{Ryzh}, as presented by de la Rue \cite[Theorem 2.4]{deLaRue}. This construction may also be viewed as a refinement of the standard averaging argument for obtaining fixed points in continuous affine actions of $\fSym(\omega)$ on compact convex sets (a special case of  Day's fixed-point characterization of amenability \cite{Day}). 
\end{remark}

\section{Morley products of generically stable measures}\label{sec:MP}

In this section we show that generically stable types are closed under Morley products in any (continuous) theory. Before getting to the proof, we first explain some history of this problem.

First, if $T$ is a discrete NIP theory then, by \cite[Proposition 3.2]{HP}, a type is generically stable if and only if it is definable and finitely satisfiable (in some small model). This immediately yields that generically stable types are closed under Morley products in this context (see also \cite[Observation 7.7]{Usvy}). The analogous characterization of \fim\ measures is proved in \cite{HPS}, and the adaptation of all of this to continuous NIP theories is done in \cite{AndGS} (see also \cite[Section 2]{KhGSmodes} for related results on generically stable types in continuous theories). 

Outside of NIP theories, the problem is more difficult since the fundamental characterization of generic stability as ``definable and finitely satisfiable" fails. In \cite{ACPgs}, Adler, Casanovas, and Pillay proposed an example of a (discrete) theory $T$ and a generically stable type $p$ such that $p\otimes p$ is not generically stable. This example was revisited by the first two authors in \cite{CoGa}, where it was further observed that the theory $T$ is NSOP$_1$ and TP$_2$. However, the third author later noticed that the type $p$ is in fact not well-defined. An explanation of the issue was given in \cite[Section 8.1]{CoGaHa} (see Remark 8.6), where we further showed that the theory $T$ in question has no nontrivial generically stable types. 

In the first draft of \cite{CoGaHa} posted to the arXiv, we claimed to prove that \fim\ measures in any discrete theory are closed under Morley products. A flaw in our argument was soon found by Silvain Rideau-Kikuchi and Paul Wang, which led to the removal of this result (see the discussion at the end of \cite[Section 6]{CoGaHa} for details). 
Then in \cite{CGH2}, we returned to the question, with a specific focus on generically stable types in discrete theories. We introduced the notion of a ``stable \emph{ict}-pattern", and showed that in a discrete theory, such a pattern characterizes the existence of two generically stable types whose Morley product is not generically stable. Using this, we proved that generically stable types are closed under Morley products in discrete NTP$_2$ theories. Around the same time, Kaplan, Ramsey, and Simon \cite{KRS} established that the same is true for \emph{treeless} theories. Beyond this, the question of preservation of generic stability under Morley products in arbitrary theories has remained open, even in the case of types. Here we will finally put the issue to rest.

We begin with a standard fact.

\begin{fact}\label{fact:commute}
Fix an integer $\ell\geq 1$ and a  model $M\prec\cU$. Suppose $p_1\in S_{x_1}(\cU),\ldots, p_\ell\in S_{x_\ell}(\cU)$ are each generically stable over $M$. Then for any set $C\supseteq M$ and any  Morley sequence $(a_{1,i}\ldots a_{\ell,i})_{i<\omega}$ in $p_1\otimes\ldots\otimes p_\ell$ over $C$, the sets $\{a_{t,i}:i<\omega\}$ for $1\leq t\leq \ell$ are mutually indiscernible over $C$.
\end{fact}
\begin{proof}
First, by generic stability, any two (not necessarily distinct) types chosen from $\{p_1,\ldots,p_\ell\}$ commute.  The claim now follows by a straightforward  exercise (cf., \cite[Proposition 4.7$(b)$]{CGH2}). 
\end{proof}

 Next we prove main technical result needed to establish that generically stable types are closed under Morley products. In particular, we show that generically stable types satisfy a strengthening of Definition \ref{def:gs} in which a Morley sequence in the type can be paired with an arbitrary convergent sequence of parameters. This result appears to be new in full generality. In Remark \ref{rem:NIP} below, we sketch a shorter proof in the NIP case using standard methods. 

\begin{lemma}\label{lem:MPtypes}
Suppose $p\in S_x(\cU)$ is generically stable over $M\prec\cU$ and $(b_i)_{i<\omega}$ is a sequence in $\cU^y$ such that $(\tp(b_i/\cU))_{i<\omega}$ converges to some $q\in S_y(\cU)$. Let $(a_i)_{i<\omega}$ be a Morley sequence in $p$ over $Mb_{<\omega}$. Then $\lim_{i\to\infty}\tp(a_ib_i/\cU)=p\otimes q$. 
\end{lemma}
\begin{proof}
Given  a tuple $z$ of variables, and an $M$-definable predicate $\theta(x,y,z)$, define
\[
\theta^n(x_0,\ldots,x_{n-1})=\tsup{yz}\min\{\theta(x_0,y,z),\ldots,\theta(x_{n-1},y,z)\}.
\]
Given $n\geq 1$, let $\mathcal{P}_n$ be the set of nonnegatively-valued $M$-definable predicates $\theta(x,y,z)$ such that $\theta^n(p^{(n)})=0$. \medskip

\noindent\textit{Claim.} For all $n\geq 1$, if $\theta(x,y,z)$ is in $\mathcal{P}_n$ then, for any $c\in\cU^z$, 
\[
\lim_{i\to\infty}\theta(a_i,b_i,c)=0.
\]

Since the proof of the claim is somewhat technical, we will first explain how it yields the lemma. In particular, fix an $\cL$-formula $\varphi(x,y,z)$ and a tuple $c\in\cU^z$. We want to show
\[
\lim_{i\to\infty}\varphi(a_i,b_i,c)=\varphi(x,y,c)^{p\otimes q}.
\]

Fix an arbitrary $\epsilon>0$. 
Let $F=F^{\varphi(x;y,z)}_p$ and define
\[
\theta(x,y,z)=|\varphi(x,y,z)-F(y,z)|\dotminus\epsilon.
\]
Note that $\theta(x,y,z)$ is a nonnegatively-valued $M$-definable predicate. Since $p$ is generically stable over $M$, it follows from \cite[Fact 3.4$(b)$]{CGH2} that there is some integer $n\geq 1$ such that $\theta^n(p^{(n)})=0$. So $\theta(x,y,z)$ is in $\mathcal{P}_n$. Thus, by the claim, 
\[
\lim_{i\to\infty}|\varphi(a_i,b_i,c)-F(b_i,c)|\dotminus\epsilon=0.
\]

Since  $\epsilon$  was arbitrary, we now have  $\lim_{i\to\infty}|\varphi(a_i,b_i,c)-F(b_i,c)|=0$.
Therefore
\[
\lim_{i\to\infty}\varphi(a_i,b_i,c)=\lim_{i\to\infty}F(b_i,c)=F(y,c)^q=\varphi(x,y,c)^{p\otimes q}.
\]

So it suffices to establish the claim.\medskip

\noindent\textit{Proof of the Claim.} We proceed by induction on $n$. For the base case, note that if $\theta(x,y,z)$ is in $\mathcal{P}_1$, then $\theta(a,b,c)=0$ for any $a\models p|_M$, $b\in\cU^y$, and $c\in\cU^z$, and hence the conclusion of the claim holds trivially. 

Assume the claim holds for some fixed $n\geq 1$ and fix some $\theta(x,y,z)$ in $\mathcal{P}_{n+1}$. Let $\xbar=(x_0,\ldots,x_{n-1})$, and let $G(\xbar)$ denote the $M$-definable predicate $F^{\theta^{(n+1)}(x;\xbar)}_p$. 
Note that if $\bar{e}\models p^{(n)}|_M$, then $G(\bar{e})=\theta^{(n+1)}(p^{(n+1)})=0$, where the final equality holds since $\theta(x,y,z)$ is in $\mathcal{P}_{n+1}$.

Now, toward a contradiction, suppose $\lim_{i\to\infty}\theta(a_i,b_i,c)\neq 0$ for some $c\in\cU^z$. After passing to a subsequence, we may fix  $\epsilon>0$ such that $\theta(a_i,b_i,c)\geq\epsilon$ for all $i<\omega$. Since $p$ is generically stable over $M$, we can apply \cite[Fact 3.4$(b)$]{CGH2} to obtain an integer $m\geq 0$ such that for any $\bar{e}\in\cU^{\xbar}$ and any Morley sequence $(a'_i)_{i<\omega}$ in $p$ over $M$, we have
\begin{equation}\label{eq:MP2}
|\{i<\omega:|\theta^{n+1}(a'_i;\bar{e})-G(\bar{e})|\geq {\textstyle\frac{\e}{2}}\}|\leq m.
\end{equation}

Set $\zbar=(z_0,\ldots,z_m)$ where $m$ is from (\ref{eq:MP2}). Consider the following $M$-definable predicates:
\begin{align*}
\psi(x,y,\zbar) &= \epsilon\dotminus \tmin{0\leq t\leq m}  \theta(x,y,z_t),\\
\chi(y,\zbar) &=\tinf{\xbar}(G(\xbar)+{\textstyle \max_{0\leq i<n}}\psi(x_i,y,\zbar)),\\
\theta_0(x,y,\zbar) &= \chi(y,\zbar)\dotminus\psi(x,y,\zbar).
\end{align*}

 We claim that $\theta_0(x,y,\zbar)$ is in $\mathcal{P}_n$. To verify this, first note that $\theta_0(x,y,\zbar)$ is nonnegatively-valued by definition. So we need to show $\theta^n_0(p^{(n)})=0$. Fix $\bar{e}\models p^{(n)}|_M$, $b\in\cU^y$, and $\cbar\in\cU^{\zbar}$. We need to find  $0\leq i^*<n$ such that $\theta_0(e_{i^*},b,\cbar)=0$. Choose $0\leq i^*<n$ such that $\psi(e_{i^*},b,\cbar)=\max_{0\leq i<n}\psi(e_i,b,\cbar)$. Recall that $G(\bar{e})=0$. Hence $\bar{e}$ witnesses that $\chi(b,\cbar)\leq \psi(e_{i^*},b,\cbar)$, and so $\theta_0(e_{i^*},b,\cbar)=0$. 
 
 Since $\theta_0(x,y,\zbar)$ is in $\mathcal{P}_n$, we can apply the induction hypothesis to conclude that for any $\cbar\in\cU^{\zbar}$,
 \begin{equation}
 \lim_{i\to\infty}\chi(b_i,\cbar)\dotminus \psi(a_i,b_i,\cbar)=0.\label{eq:MP3}
 \end{equation}

Recall that $(a_i)_{i<\omega}$ is a Morley sequence in $p$ over $M b_{<\omega}$. By Fact \ref{fact:commute} (with $\ell=1$ and $p_1=p$), $\{a_i:i<\omega\}$ is an indiscernible set over $Mb_{<\omega}$. So we may fix an automorphism $\sigma\in\Aut(\cU/Mb_{<\omega})$ such that $\sigma(a_i)=a_i$ if $i$ is even and $\sigma(a_i)=a_{i+2}$ if $i$ is odd. For $t<\omega$, set $c_t=\sigma^t(c)$. Then  for any $i,t<\omega$, we have
\begin{equation}\label{eq:MP0}
\epsilon\leq \theta(a_i,b_i,c)=\begin{cases}
\theta(a_i,b_i,c_t) & \text{if $i$ is even,}\\
\theta(a_{i+2t},b_i,c_t) & \text{if $i$ is odd.}
\end{cases}
\end{equation}

Set $\cbar=(c_0,\ldots,c_m)\in\cU^{\zbar}$. If $i<\omega$ is even then $\psi(a_i,b_i,\cbar)=0$ by (\ref{eq:MP0}), and thus  $\lim_{i\to\infty}\chi(b_{2i},\cbar)=0$ by (\ref{eq:MP3}). On the other hand, $\lim_{i\to\infty} \chi(b_i,\cbar)$ exists since $(\tp(b_i/\cU))_{i<\omega}$ converges (to $q$). Therefore $\lim_{i\to\infty}\chi(b_i,\cbar)=0$, which implies that we may fix some odd $i<\omega$ and some $\bar{e}\in\cU^{\xbar}$ such that 
\begin{equation}\label{eq:MP4}
G(\bar{e})+\max_{0\leq j<n}\psi(e_j,b_i,\cbar)<{\textstyle\frac{\epsilon}{2}}.
\end{equation} 
Since $(a_{i+2t})_{t<\omega}$ is a Morley sequence in $p$ over $M$, we can also apply (\ref{eq:MP2}) to find some  $0\leq t\leq m$ such that $\theta^{n+1}(a_{i+2t};\bar{e})<G(\bar{e})+{\textstyle\frac{\epsilon}{2}}$. 
But then we have
\begin{align*}
\epsilon-\tmax{0\leq j<n}\psi(e_j,b_i,\cbar) &\leq \min\{\epsilon,\tmin{0\leq j<n}\theta(e_j,b_i,c_t)\}\\
 &\leq \min\{\theta(a_{i+2t},b_i,c_t),\theta(e_0,b_i,c_t),\ldots\theta(e_{n-1},b_i,c_t)\}\\
  &\leq  \theta^{(n+1)}(a_{i+2t};\bar{e})\\
   &<G(\bar{e})+{\textstyle\frac{\epsilon}{2}}\\
   &< \epsilon-\tmax{0\leq j<n}\psi(e_j,b_i,\cbar),
\end{align*}
where the first inequality holds by definition of $\psi$, the second inequality holds by (\ref{eq:MP0}), the third inequality holds by definition of $\theta^{n+1}$,  the fourth inequality holds by choice of $t$, and the fifth inequality holds by (\ref{eq:MP4}).  This is a contradiction. 
\end{proof}

\begin{remark}\label{rem:NIP}
If $T$ is NIP, then Lemma \ref{lem:MPtypes} has a much shorter proof. For simplicity, we also assume $T$ is discrete. Let $p$, $(b_i)_{i<\omega}$, and $q$ be as in the statement. Fix an $\cL_{\cU}$-formula $\varphi(x,y,c)$. Define an equivalence relation $\sim$ on $\omega$ so that $i\sim i'$ if and only if, for all $j<\omega$, $\varphi(a_i,b_j,c)$ and $ \varphi(a_{i'},b_j,c)$ have the same truth value. We claim that $\sim$ has finitely many classes.\footnote{This also follows from the $n=1$ case of \cite[Theorem 3.33]{Sibook}, which is a more general result proved using honest definitions.} Indeed, suppose $(i_t)_{t<\omega}$ is an increasing sequence of pairwise $\sim$-inequivalent elements of $\omega$. For each $t<\omega$, choose $j_t<\omega$ such that $\varphi(a_{i_{2t}},b_{j_t},c)$ and $\varphi(a_{i_{2t+1}},b_{j_t},c)$ have different truth values. For $\sigma\seq\omega$ and $t<\omega$, choose $a^\sigma_t\in\{a_{i_{2t}},a_{i_{2t+1}}\}$ such that $\varphi(a^\sigma_t,b_{j_t},c)$ has the same truth value as ``$t\in\sigma$". Since $(a_i)_{i<\omega}$ is indiscernible over $b_{<\omega}$, we can find, for each $\sigma\seq\omega$, some $c_\sigma$ such that $a_{<\omega}b_{<\omega}c_\sigma\equiv a^\sigma_{<\omega}b_{<\omega}c$. So now $\varphi(a_t,b_{j_t},c_\sigma)$ holds if and only if $t\in\sigma$, contradicting NIP. 

Now fix some infinite  $\sim$-class $I\seq\omega$. For each $j<\omega$, the truth value of $\varphi(a_i,b_j,c)$ is constant on $I$. Since $p$ is generically stable, it follows that $\varphi(a_i,b_j,c)$ coincides with $F(b_j,c)$ for all $i\in I$ and $j<\omega$, where $F=F^{\varphi(x;y,z)}_p$.  Therefore $I$ is the only infinite $\sim$-class, and hence $\lim_{i\to\infty}\varphi(a_i,b_i,c)=\lim_{i\to\infty}F(b_i,c)=F(y,c)^q=\varphi(x,y,c)^{p\otimes q}$. 
\end{remark}

\begin{theorem}\label{thm:MPtypes}
Suppose $p\in S_x(\cU)$ and $q\in S_y(\cU)$ are generically stable over $M\prec\cU$. Then $p\otimes q$ is generically stable over $M$.
\end{theorem}
\begin{proof}
Let $(a_i,b_i)_{i<\omega}$ be a Morley sequence in $p\otimes q$ over $M$. Then $(b_i)_{i<\omega}$ is a Morley sequence in $q$ over $M$, and hence $\lim_{i\to\infty}\tp(b_i/\cU)=q$ since $q$ is generically stable over $M$. Moreover,  by Fact \ref{fact:commute} (with $\ell=2$, $p_1=p$, and $p_2=q$), $(a_i)_{i<\omega}$ is a Morley sequence in $p$ over $Mb_{<\omega}$. Therefore $\lim_{i\to\infty}\tp(a_ib_i/\cU)=p\otimes q$ by Lemma \ref{lem:MPtypes}. By definition, we have shown that $p\otimes q$ is generically stable over $M$.
\end{proof}

Finally, we upgrade the  result to measures using  the  connection between generically stable measures in $T$ and generically stable types in $T^R$.

\begin{proof}[\textnormal{\textbf{Proof of Theorem \ref{thm:MP}}}]
Suppose $\mu\in \kM_x(\cU)$ and $\nu\in \kM_y(\cU)$ are generically stable over $M\prec\cU$. 
By assumption and Definition \ref{def:gsm}, the types $r_\mu\in S^R_x(\Cc)$ and $r_\nu\in S^R_y(\Cc)$ are generically stable over $M^\Omega$. Therefore $r_\mu\otimes r_\nu$ is generically stable over $M^\Omega$ by Theorem \ref{thm:MPtypes} (applied in $T^R$). Thus $r_{\mu\otimes\nu}$ is generically stable over $M^\Omega$ by Proposition \ref{prop:prod-same-1}. So $\mu\otimes\nu$ is generically stable over $M$ by Definition \ref{def:gsm}.  
\end{proof}

\section{Related properties and questions}\label{sec:more}

\subsection{Order property patterns for generically stable types}
Recall  that a global type $p\in S_x(\cU)$ is generically stable over $M\prec\cU$ if and only if there is no instance of the order property for a formula $\varphi(x,y)$ involving sequences $(a_i)_{i<\omega}$ from $\cU^x$ and $(b_i)_{i<\omega}$ from $\cU^y$, where $(a_i)_{i<\omega}$ is a Morley sequence in $p$ over $M$. In \cite{CGH2}, we pursued an analogous characterization of generic stability for Keisler measures (with $T$ discrete), but without a conclusive result (see \cite[Proposition 2.6]{CGH2} and subsequent remarks). Here we establish such a characterization, which differs from  \cite{CGH2} (see Corollary \ref{cor:CGHOP} and Question \ref{ques:OP} below), but nevertheless accomplishes the desired analogy.

\begin{theorem}\label{thm:COP}
Suppose $\mu\in\kM_x(\cU)$ is Borel-definable over $M\prec\cU$. Then the following are equivalent.
\begin{enumerate}[$(i)$]
\item $\mu$ is generically stable over $M$.
\item There does not exist an $\cL$-formula $\varphi(x,y)$, a measure $\eta\in\KM_{\boldsymbol{y}}(\mu/M)$, and some $r<s$ in $\R$ such that for all $i,j<\omega$, if $i\leq j$ then $\eta(\varphi(x_i,y_j))\leq r$ and if $i>j$ then $\eta(\varphi(x_i,y_j))\geq s$.
\end{enumerate}
\end{theorem}
\begin{proof}
$(i)\Rightarrow (ii)$. Assume $(i)$ and, toward a contradiction,  suppose $(ii)$ fails witnessed by $\varphi(x,y)$, $\eta$, and $r<s$. By $(i)$ and Proposition \ref{prop:SSA}, we have that for any $\sigma\in\KM_y(\mu/M)$,
\begin{equation*}
\lim_{i\to\infty}\sigma(\varphi(x_i,y))=(\mu\otimes\sigma|_y)(\varphi(x,y)).\tag{$\ast$}
\end{equation*}

Since $\KM_y(\mu/M)$ is compact, the sequence $(\eta|_{\bfx y_j})_{j<\omega}$ in $\KM_y(\mu/M)$ has a subnet converging to some $\sigma\in\KM_y(\mu/M)$. For a fixed $i<\omega$, since $\eta(\varphi(x_i,y_j))\leq r$ for all $j\geq i$, it follows that $\sigma(\varphi(x_i,y))\leq r$. Therefore by $(\ast)$,
\[
(\mu\otimes\sigma|_y)(\varphi(x,y))=\lim_{i\to\infty}\sigma(\varphi(x_i,y))\leq r.
\]
On the other hand, for any fixed $j<\omega$, using $(\ast)$, we have
\[
\int F^\varphi_\mu\,d\eta|_{y_j}=(\mu\otimes\eta|_{y_j})(\varphi(x,y_j))=\lim_{i\to \infty}\eta(\varphi(x_i,y_j))\geq s.
\]
 Therefore, since $F^\varphi_\mu$ is continuous, 
 \[
s\leq \int F^\varphi_\mu\, d\sigma|_y=(\mu\otimes\sigma|_y)(\varphi(x,y))\leq r,
\]
which is a contradiction.\medskip

$(ii)\Rightarrow (i)$. Suppose $(i)$ fails. Then $\mu$ is not self-averaging over $M$. So there is a measure $\lambda\in\kM_{\bfx}(\cU)$, with $\lambda|_M=\mu^{(\omega)}|_M$, and an $\cL_{\cU}$-formula $\varphi(x,b)$, with $b\in\cU^y$, such that $\lim_{i\to\infty}\lambda(\varphi(x_i,b))\neq\mu(\varphi(x,b))$. After passing to a subsequence and replacing $\varphi(x,y)$ with $-\varphi(x,y)$ if necessary, we may assume that for some $\epsilon>0$, we have $\lambda(\varphi(x_i,b))\geq\mu(\varphi(x,b))+\epsilon$ for all $i<\omega$. 

Now define $\sigma\in\KM_y(\mu/M)$ such that for any $\cL_M$-formula $\theta(\bfx,y)$, $\sigma(\theta(\bfx,y))=\lambda(\theta(\bfx,b))$. 
We inductively define a sequence $(\sigma_i)_{i<\omega}$ in $\KM_y(\mu/M)$ so that $\sigma_0=\mu\otimes\sigma$ and $\sigma_{i+1}=\mu\otimes\sigma_i$.\footnote{Here we use the same abuse of notation mentioned before Proposition \ref{prop:SSA}.} 
\medskip

\noindent\textit{Claim.} There is some $\eta\in\KM_{\bfy}(\mu/M)$ such that for all $j<\omega$, $\eta|_{\bfx y_j}=\sigma_j$. 

\noindent\textit{Proof.} 
For $j<\omega$, let $\boldsymbol{a}_j c_j$ realize an extension of $\sigma_j$ in $S^R_{\bfx y_j}(M^c)$ (via Fact \ref{fact:nup}). Then for any $j<\omega$, we have $\boldsymbol{a}_j\equiv_{M^c} \boldsymbol{a}_0$, so we can choose some $c'_j$ such that $\boldsymbol{a}_jc_j\equiv_{M^c} \boldsymbol{a}_0c'_j$. Let $p=\tp(\boldsymbol{a}_0\boldsymbol{c}'/M^c)$ and let $\eta=\nu_p$. \clqed  
\medskip

Let $\eta\in\KM_{\bfy}(\mu/M)$ be as in the claim. Then for all $i,j<\omega$, if $i\leq j$ then $\eta(\varphi(x_i,y_j))=\sigma_j(\varphi(x_i,y))=\mu(\varphi(x,b))$, while if $i>j$ then 
\[
\eta(\varphi(x_i,y_j))=\sigma_j(\varphi(x_i,y))=\sigma(\varphi(x_{i-j-1},y))=\lambda(\varphi(x_{i-j-1},b))\geq\mu(\varphi(x,b))+\epsilon.
\]
Altogether, we have constructed a failure of $(ii)$.
\end{proof}

As an immediate corollary, we obtain the following extension of \cite[Proposition 2.6]{CGH2} to continuous $T$.\footnote{In \cite{CGH2}, this result is stated in terms of self-averaging.}

\begin{corollary}\label{cor:CGHOP}
Suppose $\mu\in\kM_x(\cU)$ is generically stable over $M\prec\cU$. Then there does not exist an $\cL$-formula $\varphi(x,y)$, a measure  $\lambda\in\KM(\mu/M)$, a sequence $(b_i)_{i<\omega}$ from $\cU^y$, and some $r<s$ in $\R$ such that for all $i,j<\omega$, if $i\leq j$ then $\lambda(\varphi(x_i,b_j))\leq r$ and if $i>j$ then $\lambda(\varphi(x_i,b_j))\geq s$.
\end{corollary}

\begin{question}\label{ques:OP}
Does the converse of Corollary \ref{cor:CGHOP} hold?
\end{question}

\subsection{Random generic stability}\label{sec:rgs}

In \cite{Khanaki}, Khanaki introduces the following (a priori) weaker notion of generic stability for measures.

\begin{definition}
A measure $\mu\in\kM_x(\cU)$ is \textbf{randomly generically stable (\emph{rgs}) over $M\prec\cU$} if it has an extension $p\in S^R_x(\Cc)$ which is generically stable over $M^\Omega$.
\end{definition}

Strictly speaking, this definition is formulated in \cite{Khanaki} under the assumption that $\mu$ is definable over $M\prec\cU$. However, Fact 2.11$(i)$ and Theorem 3.2[$(i)\Rightarrow(ii)$] of \cite{Khanaki} show that if $\mu$ is \emph{rgs} over $M$ then it is definable over $M$ (when $T$ is discrete). A more direct proof valid for continuous $T$ can be obtained from the purely topological criterion given by \cite[Lemma 2.3]{CGH2} (similar to the proof of definability for self-averaging measures).

In any case, Khanaki \cite{Khanaki} shows that any \emph{rgs} measure is finitely approximable, and that \emph{rgs} coincides with generic stability at the level of types (see Proposition 3.5, Theorem 3.2, and Fact 2.11 in \cite{Khanaki}). He also characterizes \emph{rgs} measures as those admitting a ``stable basic sequence" \cite[Theorem 3.2]{Khanaki}, which  roughly means that $\mu$ can be  approximated by certain finitely supported measures in a rather strong way allowing uniform control over ``random parameters". Altogether, these results establish  \emph{rgs} as a robust  generic stability-like property. This leaves open the following natural question, which is due to Khanaki \cite[Question 3.9]{Khanaki}.

\begin{question}\label{ques:rgs}
Suppose $\mu\in\kM_x(\cU)$ is \emph{rgs} over $M\prec\cU$. Is $\mu$ generically stable over $M$?\footnote{In \cite{Khanaki}, generic stability is phrased in terms of ``\emph{irgs}", i.e., condition $(ii)$ of Theorem \ref{thm:main}.}
\end{question}

\subsection{Final note}\label{sec:probs} Given the significant role played by AI models in the above results, the reader may wonder about  the situation  with Questions \ref{ques:OP} and \ref{ques:rgs}. In the interest of full transparency, we feel it is appropriate to disclose our findings in this direction. ChatGPT 5.6 did not successfully produce answers to either question. However, a later query with ChatGPT 6 resulted in proposed proofs of positive answers to both. As of writing, the authors have not  fully checked the details of these proofs, which are both fairly involved. Nor have we  done the due diligence of finding the proper attribution for the underlying arguments. Moreover, it appears that a presentation these arguments would require a major overhaul of the present paper in order to be done efficiently and with the proper clarity. For this reason, we have decided to not pursue these lines of investigation here. Both questions are left open for future research.

\appendix 

\section{Proofs for continuous logic}\label{sec:appendix}

\subsection{Proof of Fact \ref{fact:rtrans}}\label{sec:randoproofs}

In this subsection, we go through each result and cited fact in \cite[Section 3.2]{CGH2} and note only which lines in the proofs from \cite{CGH2} need adjustment to extend from discrete $T$ to continuous $T$. Throughout the section, we make parenthetical references to \cite{CGH2} in each statement to indicate the corresponding result for \emph{discrete} $T$.  The first result, namely \cite[Fact 3.7]{CGH2}, is the quantifier elimination property for $T^R$ that we have already restated for continuous $T$ in Fact \ref{fact:randomizationsQE} above. So we will begin with \cite[Lemma 3.8]{CGH2}.

\begin{lemma}[{\cite[Lemma 3.8]{CGH2}}]\label{lem:randomization-of-continuous-function} For any continuous function $F\colon S_{x}(\Uc) \to \Rb$ (with $x$ finite), there is a unique continuous function $\Eb[F(-)]\colon S^R_{x}(\Uc^\Omega) \to \Rb$ satisfying that for any $\hbarr \in (\Uc_0^\Omega)^x$,
      \[
    \Eb[F(\tp(\hbarr/\Uc^\Omega))] = \sum_{A \in \Ac} \mathbb{P}(A)F(\tp(\hbarr|_A/\Uc)), 
      \]
      where $\Ac$ is a finite measurable partition of $\Omega$ on which the elements of $\hbarr$ are constant.
\end{lemma}
\begin{proof}
Existence is immediate from the description of $\Eb[F]$ following Fact \ref{fact:nup}, and uniqueness follows from density of $\cU_0^\Omega$ in $\cU^\Omega$. The explicit proof in \cite{CGH2} also adapts directly to continuous $T$ by replacing indicator functions of discrete formulas  by continuous  formulas (viewed as functions on type spaces).
\end{proof}

The next fact is similar to the  ``perfect witness" property for full models of randomizations, which is established for discrete $T$ in \cite[Proposition 2.5]{BYKR}.  A detailed proof  is given in \cite{CGH2}, which works verbatim for continuous $T$. Alternatively, this fact can be easily derived from the $\sup$ version of \cite[Lemma 3.13]{BYRV}. 

\begin{fact}[{\cite[Fact 3.9]{CGH2}}]\label{fact:randomizationsb}
 For any continuous map $\varphi(x,y):S_{xy}(\mathcal{U}) \to [0,1]$, 
\begin{equation*}
    \mathcal{U}^{\Omega} \models \forall {x}\left(\textstyle \sup_{{y}} \Eb[\varphi({x},{y})] = \Eb\left[\sup_{{y}} \varphi({x},{y})\right] \right).
\end{equation*}
\end{fact}

The next statement combines  \cite[Fact 3.10]{CGH2} and \cite[Remark 3.11]{CGH2}, which  together yield everything in Fact \ref{fact:rtrans} except for the moreover statement.

\begin{fact}[{\cite[Fact 3.10]{CGH2} and \cite[Remark 3.11]{CGH2}}]\label{fact:unique-definable-transfer}
   Let $\mu\in \mathfrak{M}_x(\Uc)$ be definable over $M\prec\cU$. There is a unique $M^\Omega$-definable type $r_\mu \in S^R_x(\Cc)$ defined as follows: for any $\varphi(x,\ybar) \in \Lc$, $F^{\Eb[\varphi(x,\ybar)]}_{r_\mu}= \Eb[F^{\varphi(x,\ybar)}_{\mu}]$.  
\end{fact}
\begin{proof}
The proof of \cite[Fact 3.10]{CGH2} works verbatim when $T$ is continuous. This establishes the existence and uniqueness of $r_\mu$, but only definability over $\mathcal{U}^\Omega$. Remark 3.11 of \cite{CGH2} then shows that $r_\mu$ is also $M^\Omega$-invariant (and hence definable over $M^\Omega$). For continuous $T$, one only needs to adjust the middle paragraph as follows.

Since $\mu$ is definable over $M$, for every $\epsilon > 0$, there exists an $\mathcal{L}_M$-formula $\psi_\epsilon(y,b)$ such that 
\begin{equation*} 
\sup_{q \in S_{y}(\mathcal{U})}\left|F_{\mu}^{\varphi}(q) -  \psi(q,b)\right| < \epsilon. 
\end{equation*}  
The $G_{\epsilon}$ in the proof of Lemma \ref{lem:randomization-of-continuous-function} (adapted directly from the proof of \cite[Lemma 3.8]{CGH2} as indicated above) can be constructed using this $\psi(y,b)$. Observe, 
\begin{multline*} 
\widehat{G}_{\epsilon}(\tp(h/\mathcal{U}^{\Omega})) =  (\mathbb{E}[\psi(y,f_{b})])^{\tp(h/\mathcal{U}^{\Omega})} =  (\mathbb{E}[\psi(h,f_{b})]) \\ 
=(\mathbb{E}[\psi(g,f_{b})]) = (\mathbb{E}[\psi(y,f_{b})])^{\tp(g/\mathcal{U}^{\Omega})} = \widehat{G}_{\epsilon}(\tp(g/\mathcal{U}^{\Omega})). 
\end{multline*} 
Other than this change, the rest of the argument is identical. 
\end{proof}

The next corollary in \cite{CGH2} follows immediately from the fact that $r_\mu$ extends $\mu$ and the discussion after Fact \ref{fact:nup}. (The direct proof in \cite{CGH2} also has a clear translation to continuous $T$.)

\begin{corollary}[{\cite[Corollary 3.12]{CGH2}}]\label{cor:rmu-def} 
Let $\mu \in \mathfrak{M}_x(\mathcal{U})$ be a definable measure. Then for any continuous function $f\colon S_{x}(\mathcal{U}) \to \mathbb{R}$, we have  $\int f\, d\mu = (\Eb[f])^{r_{\mu}}$. 
\end{corollary}

\begin{lemma}[{\cite[Lemma 3.13]{CGH2}}]\label{random:1} 
Let $\mu \in \mathfrak{M}_{x}(\mathcal{U})$ be definable. Let $(h_i)_{i \in I}$ be a net in $(\cU_0^{\Omega})^{x}$ such that $\lim_{i \in I} \tp(h_i/\mathcal{U}^{\Omega}) = r_{\mu} |_{\Uc^\Omega}$. For each $i \in I$,  let $\mathcal{D}_i$ be a finite measurable partition of $\Omega$ such that $h_i$ is constant on each $D \in \mathcal{D}_i$. Set $\mu_i = \sum_{D \in \mathcal{D}_i} \mathbb{P}(D) \delta_{h_i|_{D}}$. Then $\lim_{i \in I} \mu_i = \mu$.
\end{lemma} 
\begin{proof}
In the continuous setting, the first two displayed equations naturally condense into one: 
\[
\lim_{i \in I} \mu_i(\varphi(x,\bbar)) = \lim_{i \in I}
  \sum_{D \in \mathcal{D}_i} \mathbb{P}(D)\varphi(h_i|_{D},\bbar).
  \]
Then after the third and fourth equation symbols, replace the two instances of
\[
\mathbb{P}(\{t \in \Omega: \Uc \models \varphi(h_i(t),-) \}) \quad \text{with} \quad \int_\Omega \varphi(h_i(t),-)\,d\mathbb{P}(t).\qedhere
\]
\end{proof}

\begin{lem}[{\cite[Lemma 3.14]{CGH2}}]\label{lem:independent-limit} Fix a definable measure $\mu \in \mathfrak{M}_{x}(\mathcal{U})$ and a finite, measurable partition $\mathcal{A}$ of $\Omega$. Then there is a net $(h_i)_{i \in I}$ of elements in $(\Uc^\Omega_0)^{x}$    and finite measurable partitions $(\Dc_i)_{i \in I}$ of $\Omega$ satisfying the following properties:
\begin{enumerate}[$(i)$]
    \item $\lim_{i \in I} \tp(h_i/\Uc^\Omega) = r_{\mu}|_{\mathcal{U}^{\Omega}}$. 
    \item $h_i$ is constant on each element of $\mathcal{D}_i$.
    \item  $\mathbb{P}(A \cap D) = \mathbb{P}(A) \mathbb{P}(D)$ for each $A \in \mathcal{A}$ and $D \in \mathcal{D}_i$. 
\end{enumerate}
\end{lem}
\begin{proof}
In the final displayed computations, after the first three equation symbols replace the instances of $\boldsymbol{1}_{\varphi_i}(-,-)$ with $\varphi_i(-,-)$. 
\end{proof}

\begin{prop}[{\cite[Proposition 3.15]{CGH2}}]\label{prop:prod-same-1}
    If $\mu\in\kM_x(\cU)$ and $\nu\in\kM_y(\cU)$ are definable, then $ r_{\mu\otimes \nu}(x,y) = r_\mu(x)\otimes r_\nu(y)$.
\end{prop}
\begin{proof}
The  proof works verbatim for continuous $T$. One only needs to account for the background fact that Morley products preserve definability, which is proved for continuous logic in \cite[Lemma 2.9]{AndGS} (as discussed at the beginning of Section \ref{sec:rando}). 
\end{proof}

The previous fact then yields the following corollary by induction.

\begin{corollary}[{\cite[Corollary 3.16]{CGH2}}]\label{cor:prod-same}
  Suppose  $\mu \in \mathfrak{M}_{x}(\mathcal{U})$  is definable. Then for every $n\geq 1$, $r_{\mu^{(n)}}(\bar{x}) = (r_\mu)^{(n)}(\bar{x})$.
\end{corollary}

This corollary supplies the final moreover statement in Fact \ref{fact:rtrans}.

\subsection{Proof of Proposition \ref{prop:ABC}}\label{sec:ABC}
When $T$ is discrete, Proposition \ref{prop:ABC} is nearly identical to \cite[Proposition 3.18]{CGH2}, except that we have been more specific about the base model $M$. In any case, with Subsection \ref{sec:randoproofs} in hand for continuous $T$, one can now follow the proof of \cite[Proposition 3.18]{CGH2} verbatim to obtain the proof of Proposition \ref{prop:ABC}.

\subsection{Proof of Lemma \ref{lem:SAfam}}

Recall that Lemma \ref{lem:SAfam} states that self-averaging measures are finitely approximable and thus, in particular, definable and self-commuting. This is proved for discrete $T$ in \cite{CGH2} by first establishing definability (see \cite[Corollary 2.4]{CGH2}), and then finite approximability (see \cite[Theorem 2.7$(b)$]{CGH2}). Almost every line in the proofs of these results can be copied verbatim in the case that $T$ is continuous. The only  points requiring comment are as follows.

\begin{enumerate}[$(1)$]
\item The proof of \cite[Corollary 2.4]{CGH2} contains a claim asserting that a certain set $C$ is closed in $[0,1]^\omega\times S_y(M)$. In the proof of this claim, replace  ``a $\psi(y)\in\tp(b/M)$'' with ``an open set $U$ containing $\tp(b/M)$", and then replace the two instances of ``$c\in\psi(\cU)$'' with ``$c$ such that $\tp(c/M)\in U$''.
\item The proof of \cite[Lemma 2.9]{CGH2} (one ingredient needed for \cite[Theorem 2.7$(b)$]{CGH2}) is written tersely since it is completely standard  in discrete logic. The continuous analogue was already addressed in Lemma \ref{lem:weird-extension}. 
\item The conclusion that self-averaging measures are self-commuting is explained in \cite[Remark 2.11]{CGH2} using \cite[Proposition 5.17]{CoGaHa}, which states that finitely approximable measures commute with definable measures (in discrete logic). The proof of \cite[Proposition 5.17]{CoGaHa} can be rewritten for continuous logic, but the translation would require some amount of work. Another proof of this result appears in \cite[Corollary 3.8]{GanSA} using different techniques that are  perhaps more readily adaptable to continuous logic. However, for our purposes, it suffices to just directly prove that any finitely approximable measure is self-commuting. This is a much easier exercise (see, e.g., \cite[Proposition 2.10$(b)$]{CoGa} and \cite[Corollary 2.32]{Ganthesis} whose proofs translate directly to continuous logic). 
\end{enumerate}

\subsection{Further discussion on generic stability of $r_\mu$}\label{sec:constants}

In this section, we revisit Remark 3.21 of \cite{CGH2}. In light of the  equivalence between generic stability for $r_\mu$ and self-averaging for $\mu$ established here, this remark asserts that to show $r_\mu$ is generically stable it suffices to prove generic stability with respect to basic $\cL^R$-formulas with parameters in $\cU^c$. We formulate this precisely in the next result.

\begin{proposition}\label{prop:SAgs}
Suppose $\mu\in\kM_x(\cU)$ is definable over $M\prec\cU$. Assume that for any $\cL_M$-formula $\varphi(x,y)$ and any $b\in (\cU^c)^y$,  if $(a_i)_{i<\omega}$ is a Morley sequence in $r_\mu$ over $M^\Omega$ then 
\[
\lim_{i\to\infty}\Eb[\varphi(a_i,b)]=(\Eb[\varphi(x,b)])^{r_\mu}.
\]
Then $r_\mu$ is generically stable over $M^\Omega$. 
\end{proposition}

Our  results above have not fully established this in the continuous case. The missing ingredient is the following lemma, which is likely known to experts but does not appear to be written explicitly in the literature. When $T$ is discrete, this result is implicit in the proof of \cite[Fact 3.4]{Khanaki}.  
Our proof for continuous $T$ will take a more direct elementary approach.

\begin{lemma}\label{lem:constanttypes}
Fix $M\prec\cU$, $a\in \cU^x$ and $b\in\mathcal{C}^x$. Then 
\[
\tilde{a} \equiv_{M^\Omega} b \quad\text{ if and only if }\quad \tilde{a} \equiv_{M^c} b.
\]
\end{lemma}

\begin{proof}
The forward implication is immediate. For the converse, assume
$\tilde{a} \equiv_{M^c} b$. 
By quantifier elimination for $T^R$, density of $M^\Omega_0$ in $M^\Omega$, and uniform continuity, it suffices to show that for every $\cL$-formula $\varphi(x,y)$ and every $f\in (M^\Omega_0)^y$, we have
\[
\Eb[\varphi(b,f)]=\Eb[\varphi(\tilde{a},f)].
\tag{1}
\]
So fix such $\varphi$ and $f$. Let $\mathcal{A}$ be a finite measurable partition of $\Omega$ such that, for all $A\in\mathcal{A}$,  $f|_A$ takes a constant value $m_A\in M^y$. For $A\in\mathcal{A}$, set $r_A\coloneqq\varphi(a,m_A)\in \R$.  
Since $\tilde{a} \equiv_{M^c} b$, for every $A\in\mathcal A$, we have
\[
\Eb[|\varphi(b,\widetilde m_A)-r_A|]=\Eb[|\varphi(\widetilde a,\widetilde m_A)-r_A|]
=
0.
\tag{2}
\]

Fix $\epsilon>0$. Since $M^\Omega$ is an elementary substructure of $\mathcal C$, and $M^\Omega_0$ is dense in $M^\Omega$, we may choose $d\in(M^\Omega_0)^x$ such that
\[
|\Eb[\varphi(d,f)]-\Eb[\varphi(b,f)]|<\epsilon
\tag{3}
\]
and, via (2), for every $A\in\mathcal{A}$,
\[
\Eb[|\varphi(d,\widetilde m_A)-r_A|]<\frac{\epsilon}{|\mathcal{A}|}.
\tag{4}
\]
Since $f(t)=m_A$ for all $t\in A$, we have
\begin{align*}
|\Eb[\varphi(d,f)]-\Eb[\varphi(\tilde a,f)]|
&=\left|\sum_{A\in\mathcal{A}} \int_A\big(\varphi(d(t),m_A)-r_A\big)\,d\mathbb{P}(t)\right|\\
&\leq\sum_{A\in\mathcal{A}} \int_A |\varphi(d(t),m_A)-r_A|\,d\mathbb{P}(t)\\
&\leq\sum_{A\in\mathcal{A}}\Eb[|\varphi(d,\widetilde m_A)-r_A|]\\
&<\epsilon,
\end{align*}
where the final inequality uses (4).
Together with (3), this gives
\[
|\Eb[\varphi(b,f)]-\Eb[\varphi(\tilde{a},f)]|<2\epsilon.
\]
As $\epsilon>0$ was arbitrary, (1) follows.
\end{proof}

\begin{proof}[Proof of Proposition \ref{prop:SAgs}]
Let $\mu\in \kM_x(\cU)$ be as in the statement. It suffices to show that $\mu$ is self-averaging over $M$. So fix some $\lambda\in \KM(\mu/M)$ and an $\cL_{\cU}$-formula $\varphi(x,b)$, with $b\in \cU^y$. As in the proof of Proposition \ref{prop:SSA}$[(i)\Rightarrow(ii)]$, we define $p=r^{(\omega)}_\mu|_{M^\Omega}$, which extends $\mu^{(\omega)}|_M=\lambda|_M$. Define $\sigma\in\kM_{\bfx y}(M)$ such that, given an $\cL_M$-formula $\theta(\bfx,y)$, $\sigma(\theta(\bfx,y))=\lambda(\theta(\bfx,b))$. Note that $p$ extends $\sigma|_{\bfx}$. By Lemma \ref{lem:measure-extension}, there is some $q\in S^R_{\bfx y}(M^\Omega)$ such that $q|_{\bfx}=p$ and $q$ extends $\sigma$. Let $(a'_i)_{i<\omega}b'\models q$. Since $q|_y$ extends $\sigma|_y=\tp(b/M)$, it follows that $b'\equiv_{M^c} \tilde{b}$. By Lemma \ref{lem:constanttypes}, $b'\equiv_{M^\Omega} \tilde{b}$. Choose $(a_i)_{i<\omega}$ such that $(a_i)_{i<\omega}\tilde{b}\equiv_{M^\Omega}(a'_i)_{i<\omega}b'$. Then $(a_i)_{i<\omega}\tilde{b}\models q$, which implies  $(a_i)_{i<\omega}$ is a Morley sequence in $r_\mu$ over $M^\Omega$ (recall $q|_{\bfx}=p$). One can now verify that $\lim_{i\to\infty}\lambda(\varphi(x_i,b))=\mu(\varphi(x,b))$ using essentially the same steps as in the proof of Proposition \ref{prop:SSA}$[(i)\Rightarrow(ii)]$. 
\end{proof}


\begin{thebibliography}{1}

\bibitem{ACPgs}
H. Adler, E. Casanovas, and A. Pillay, \emph{Generic stability and stability},
  J. Symb. Log. \textbf{79} (2014), no.~1, 179--185. \MR{3226018}

\bibitem{AndGS}
A. Anderson, \emph{Generically stable measures and distal regularity in
  continuous logic}, arXiv:2310.06787, 2023.


\bibitem{BYT}
I. Ben Yaacov, \emph{Transfer of properties between measures and random types}, unpublished research note, \url{http://math.univ-lyon1.fr/~begnac/articles/MsrPrps.pdf}. 


  
  \bibitem{BYRV}
I. Ben~Yaacov, \emph{On theories of random variables}, Israel J. Math.
  \textbf{194} (2013), no.~2, 957--1012. \MR{3047098}


\bibitem{BYKR}
I. Ben Yaacov and H. J. Keisler, \emph{Randomizations of models as metric structures}, Confluentes Math. \textbf{1} (2009), no. 2, 197--223. \MR{2561997}



\bibitem{BRZK}
V. Bergelson, D. Robertson, and P. Zorin-Kranich, \emph{Triangles in
  {C}artesian squares of quasirandom groups}, Combin. Probab. Comput.
  \textbf{26} (2017), no.~2, 161--182. \MR{3603962}



\bibitem{CGK}
A. Chernikov, K. Gannon, and K. Krupi{\'n}ski, \emph{Definable convolution and
  idempotent {K}eisler measures {III}. {G}eneric stability, generic
  transitivity, and revised {N}ewelski's conjecture},  J. Lond. Math. Soc., published online June 24, 2026.



\bibitem{CoGa}
G. Conant and K. Gannon, \emph{Remarks on generic stability in independent
  theories}, Ann. Pure Appl. Logic \textbf{171} (2020), no.~2, 102736, 20.
 
 \bibitem{CoGaHa}
 G. Conant, K. Gannon, and J. Hanson, \emph{Keisler measures in the wild}, Model Theory \textbf{2} (2023), no.~1, 1--67.
 
\bibitem{CGH2}
\bysame, \emph{Generic stability, randomizations, and {NIP} formulas},
J. Math. Log. (2025), Paper No. 2550016, 45.

\bibitem{Day}
M. M. Day, \emph{Fixed-point theorems for compact convex sets}, Illinois J.
Math. \textbf{5} (1961), 585--590. \MR{0138100}

\bibitem{deLaRue}
T. de la Rue, \emph{An introduction to joinings in ergodic theory}, Discrete
Contin. Dyn. Syst. \textbf{15} (2006), no. 1, 121--142. \MR{2191388}



\bibitem{DitEif}
S.~Z. Ditor and L.~Q. Eifler, \emph{Some open mapping theorems for measures},
  Trans. Amer. Math. Soc. \textbf{164} (1972), 287--293. \MR{477729}

\bibitem{Ganthesis}
K. Gannon, \emph{Approximation theorems for {K}eisler measures}, Ph.D. thesis,
  University of Notre Dame, 2020.

\bibitem{GanSA}
\bysame, \emph{Sequential approximations for types and {K}eisler measures},
  Fund. Math. \textbf{257} (2022), no.~3, 305--336. \MR{4416017}

\bibitem{Gannon-note}
\bysame, \emph{On transfer maps and the Morley product in {NIP} theories},
J. Symb. Log. (2026), 1--24, published online.

\bibitem{GHevents}
K. Gannon and J. E. Hanson, \emph{Model theoretic events}, Ann. Pure
Appl. Logic \textbf{178} (2027), no. 1, Paper No. 103825.



\bibitem{HewSav}
E. Hewitt and L. J. Savage, \emph{Symmetric measures on {C}artesian
products}, Trans. Amer. Math. Soc. \textbf{80} (1955), no.~2, 470--501.
\MR{0076206}

\bibitem{Hoover}
D. N. Hoover, \emph{Relations on probability spaces and arrays of random
variables}, preprint, Institute for Advanced Study, Princeton, 1979.

\bibitem{HP}
E. Hrushovski and A. Pillay, \emph{On {NIP} and invariant measures}, J. Eur.
Math. Soc. (JEMS) \textbf{13} (2011), no. 4, 1005--1061. \MR{2800483}



    \bibitem{HPS}
E. Hrushovski, A. Pillay, and P. Simon, \emph{Generically stable and smooth
  measures in {NIP} theories}, Trans. Amer. Math. Soc. \textbf{365} (2013),
  no.~5, 2341--2366. \MR{3020101}




\bibitem{KRS}
I. Kaplan, N. Ramsey, and P. Simon, \emph{Generic stability independence and
treeless theories}, Forum Math. Sigma \textbf{12} (2024), Paper No. e49.


\bibitem{KhGSmodes}
K. Khanaki, \emph{Generic stability and modes of convergence}, J. Symb. Log.
  \textbf{91} (2026), no.~3, 1108--1132.

\bibitem{Khanaki}
\bysame, \emph{{K}eisler measures and generically stable random types},
arXiv:2605.15870, 2026.

\bibitem{Matus}
F. Mat{\'u}{\v{s}}, \emph{Finite partially exchangeable arrays}, Research
Report No. 1856, Institute of Information Theory and Automation, Academy of
Sciences of the Czech Republic, Prague, 1995.



\bibitem{phelps-book}
R.~R. Phelps, \emph{Lectures on {C}hoquet's theorem}, 2nd ed., Lecture Notes in
  Mathematics, vol. 1757, Springer, Berlin, Heidelberg, 2001.

\bibitem{PiTa}
A. Pillay and P. Tanovi\'c, \emph{Generic stability, regularity, and
  quasiminimality}, Models, logics, and higher-dimensional categories, CRM
  Proc. Lecture Notes, vol.~53, Amer. Math. Soc., Providence, RI, 2011,
  pp.~189--211.

\bibitem{Ryzh}
V. V. Ryzhikov, \emph{Joinings, intertwining operators, factors, and mixing
properties of dynamical systems}, Russian Acad. Sci. Izv. Math. \textbf{42}
(1994), no. 1, 91--114. \MR{1220583}

\bibitem{Sibook}
P. Simon, \emph{A guide to {NIP} theories}, Lecture Notes in Logic, vol.~44,
  Association for Symbolic Logic, Chicago, IL; Cambridge Scientific Publishers,
  Cambridge, 2015. 

\bibitem{Usvy}
A. Usvyatsov, \emph{On generically stable types in dependent theories}, J.
Symb. Log. \textbf{74} (2009), no. 1, 216--250. \MR{2499428}

\end{thebibliography}
\end{document}